\documentclass[reqno]{amsart}
\usepackage{amssymb}
\usepackage{amsfonts}
\usepackage{amsmath}
\usepackage{color}
\usepackage{amsaddr}

\newtheorem{theorem}{Theorem}[section]

\newtheorem{corollary}[theorem]{Corollary}

\newtheorem{definition}[theorem]{Definition}

\newtheorem{lemma}[theorem]{Lemma}

\newtheorem{proposition}[theorem]{Proposition}

\theoremstyle{remark}
\newtheorem{remark}[theorem]{Remark}

\numberwithin{equation}{section}
\newcommand{\ep}{\varepsilon}
\newcommand{\dist}{ {\mathrm{dist}} }
\newcommand{\diff}{ {\mathrm{d}} }

\newcommand{\trD}{ {\mathrm{tr_D}} }

\begin{document}
\title[Wave Equations with Robin--Acoustic Perturbation]{Attractors for Damped Semilinear Wave Equations with a Robin--Acoustic Boundary Perturbation}
\author[J. L. Shomberg]{Joseph L. Shomberg}
\address{Department of Mathematics and Computer Science, Providence College, Providence, Rhode Island 02918, USA, \\ {\tt{jshomber@providence.edu}}}
%\email{jshomber@providence.edu}
\date\today

\begin{abstract}
Under consideration is the damped semilinear wave equation
\[
u_{tt}+u_t-\Delta u + u + f(u)=0
\]
on a bounded domain $\Omega$ in $\mathbb{R}^3$ with a perturbation parameter $\ep>0$ occurring in an acoustic boundary condition, limiting ($\ep=0$) to a Robin boundary condition. 
With minimal assumptions on the nonlinear term $f$, the existence and uniqueness of global weak solutions is shown for each $\ep\in[0,1]$. 
Also, the existence of a family of global attractors is shown to exist. 
After proving a general result concerning the upper-semicontinuity of a one-parameter family of sets, the result is applied to the family of global attractors. Because of the complicated boundary conditions for the perturbed problem, fractional powers of the Laplacian are not well-defined; moreover, because of the restrictive growth assumptions on $f$, the family of global attractors is obtained from the asymptotic compactness method developed by J. Ball for generalized semiflows.
With more relaxed assumptions on the nonlinear term $f$, we are able to show the global attractors possess optimal regularity and prove the existence of an exponential attractor, for each $\ep\in[0,1].$
This result insures that the corresponding global attractor inherits finite (fractal) dimension, however, the dimension is {\em{not}} necessarily uniform in $\ep$. 
\end{abstract}

\maketitle
%\tableofcontents

%%%%%%%%%%%%%%%%%%%%%%%%%%%%%%%%%%%%%%%%%%%%%%
%%%%%%%%%%%%%%%%%%%%%%%%%%%%%%%%%%%%%%%%%%%%%%
\section{Introduction to the model problems}  \label{s:intro}
%%%%%%%%%%%%%%%%%%%%%%%%%%%%%%%%%%%%%%%%%%%%%%
%%%%%%%%%%%%%%%%%%%%%%%%%%%%%%%%%%%%%%%%%%%%%%

Let $\Omega\subset\mathbb{R}^3$ be a bounded domain with smooth boundary $\Gamma$. The equation under consideration in the unknown $u=u(t,x)$ is the semilinear damped wave equation 
\begin{equation}\label{damped-wave-equation}
u_{tt} + u_t-\Delta u + u + f(u) = 0 \quad \text{in} \quad (0,\infty)\times\Omega
\end{equation}
with the initial conditions
\begin{equation}  \label{Robin-initial-conditions}
u(0,\cdot) = u_0, ~u_t(0,\cdot) = u_1 \quad \text{on} \quad \{0\}\times \Omega
\end{equation}
and the homogeneous Robin boundary condition
\begin{equation}  \label{Robin-boundary}
\partial_{\bf{n}}u + u = 0 \quad \text{on} \quad (0,\infty)\times\Gamma,
\end{equation}
where $\bf{n}$ is the outward pointing unit vector normal to the surface $\Gamma$ at $x$, and $\partial_{\bf{n}}u$ denotes the normal derivative of $u$, $\nabla u\cdot \bf{n}$. Assume the nonlinear term $f\in C^1(\mathbb{R})$ satisfies the growth condition
\begin{equation}  \label{f-assumption-1}
|f'(s)|\leq \ell(1+s^2)
\end{equation}
for some $\ell\geq 0$, and the sign condition 
\begin{equation}\label{f-assumption-2}
\liminf_{|s|\rightarrow\infty}\frac{f(s)}{s} > -1
\end{equation}
Collectively, denote the IBVP (\ref{damped-wave-equation})-(\ref{Robin-boundary}), with (\ref{f-assumption-1})-(\ref{f-assumption-2}), or (\ref{f-assumption-2}), (\ref{f-reg-ass-2})-(\ref{f-reg-ass-3}) as Problem (R).

Also of interest is the following ``relaxation'' of the Robin boundary condition, where, for $\delta=\delta(t,x)$ and $\ep>0$, the acoustic perturbation boundary condition is employed
\begin{equation}  \label{acoustic-boundary}
\left\{ \begin{array}{ll}
\delta_{tt} + \ep[\delta_t + \delta + g(\delta)] = -u_t & \quad \text{on} \quad (0,\infty)\times\Gamma \\
\delta_t = \partial_{\bf{n}}u & \quad \text{on} \quad (0,\infty)\times\Gamma,
\end{array}\right.
\end{equation}
supplemented with the additional initial conditions
\begin{equation}\label{acoustic-initial-conditions}
\ep\delta(0,\cdot) = \ep\delta_0, \quad \delta_t(0,\cdot) = \delta_1 \quad \text{on} \quad \{0\}\times \Gamma.
\end{equation}
During the discussion on well-posedness and the dissipative qualitative behavior of   the perturbation problem, we assume $g\in C^1(\mathbb{R})$ satisfies the growth condition
\begin{equation}\label{g-assumption-1}
|g'(s)|\leq \rho
\end{equation}
for some $\rho\geq 0$, and the sign condition
\begin{equation}\label{g-assumption-2}
\liminf_{|s|\rightarrow\infty}\frac{g(s)}{s}>-1.
\end{equation}
Denote the IBVP (\ref{damped-wave-equation})-(\ref{Robin-initial-conditions}), (\ref{acoustic-boundary})-(\ref{acoustic-initial-conditions}) with (\ref{f-assumption-1})-(\ref{f-assumption-2}) or (\ref{f-assumption-2}), (\ref{f-reg-ass-2})-(\ref{f-reg-ass-3}) (below), and (\ref{g-assumption-1})-(\ref{g-assumption-2}), or (\ref{g-reg-ass-1}) (below) by Problem (A). 
During portions of this article, we will make different assumptions on the nonlinear terms $f$ and $g$.
When we discuss global attractors for Problem (R) and Problem (A), we will take 
\begin{equation}  \label{g-reg-ass-1}
g\equiv0.
\end{equation}
Also, when we discuss regularity results or exponential attractors, we will need to require more from the nonlinear term $f$; in that case, we assume the growth condition
\begin{equation}\label{f-reg-ass-2}
|f''(s)|\leq \ell_1(1+|s|)
\end{equation}
and the bound
\begin{equation}\label{f-reg-ass-3}
f'(s)\geq -\ell_2,
\end{equation}
holds for some $\ell_1,\ell_2\geq 0$, in addition to (\ref{f-assumption-2}).

The idea of ``relaxing'' a Robin boundary condition into an acoustic boundary condition comes from \cite{Beale&Rosencrans74}. 
In the singular case of the acoustic boundary condition; that is, when $\ep=0$ in (\ref{acoustic-boundary}), the following Robin boundary condition in $u_t$ is obtained,
\begin{equation}  \label{Robin-boundary-ut}
\frac{\partial^2 u}{\partial t \partial{\bf{n}}} + u_t = 0 \quad \text{on} \quad (0,\infty)\times\Gamma.
\end{equation}
The condition in equation (\ref{Robin-boundary-ut}) can be expressed as the system on $(0,\infty)\times\Gamma$,
\begin{equation}  \label{key-2}
\left\{ \begin{array}{l} \delta_{tt} = -u_t \\ \delta_t = \partial_{\bf{n}}u. \end{array} \right.
\end{equation}

In general, the damped wave equation (\ref{damped-wave-equation}) has several applications to physics, such as relativistic quantum mechanics (cf. e.g.\cite{Babin&Vishik92,Temam88}).
Problem (R) may be used to govern a thermoelastic medium with (\ref{Robin-boundary}) indicating that the amount of heat is proportional to its flux at the boundary. 
For this reason, the Robin boundary condition is also called a convection surface condition as well.
One way to motivate the Robin boundary condition may be as a static condition based on the dynamic boundary condition,
\begin{equation}  \label{DBC}
\partial_{\bf{n}}u+u+u_t=0.
\end{equation}
The dynamic term $u_t$ that now appears can be used to account for frictional forces that produce damping along the boundary on the physical domain.
The condition (\ref{DBC}) also has motivations coming from thermodynamics and also appears in connection with the Wentzell boundary condition (cf. \cite{Gal12-2} and the references therein).

Problem (A) describes a gas experiencing irrotational forces from a rest state in a domain $\Omega$. 
The surface $\Gamma$ acts as a locally reacting spring-like mechanism in response to excess pressure in $\Omega$. 
The unknown $\delta =\delta (t,x)$ represents the {\em{inward}} ``displacement'' of the boundary $\Gamma$ reacting to a pressure described by $-u_{t}$. 
The first equation (\ref{acoustic-boundary})$_{1}$ describes the spring-like effect in which $\Gamma $ (and $\delta $) interacts with $-u_{t}$, and the second equation (\ref{acoustic-boundary})$_{2}$ is the continuity condition: velocity of the boundary displacement $\delta $ agrees with the normal derivative of $u$. 
The presence of the term $g$ indicates nonlinear effects in the damped oscillations occurring on the surface. 
Together, (\ref{acoustic-boundary}) describe $\Gamma $ as a so-called locally reactive surface. 
In applications the unknown $u$ may be taken as a velocity potential of some fluid or gas in $\Omega $ that was disturbed from its equilibrium. 
The acoustic boundary condition was rigorously described by Beale and Rosencrans in \cite{Beale76,Beale&Rosencrans74}. 
Various recent sources investigate the wave equation equipped with acoustic boundary conditions, \cite{CFL01,GGG03,Mugnolo10,Vicente09}. 
However, more recently, it has been introduced as a dynamic boundary condition for problems that study the asymptotic behavior of weakly damped wave equations, see \cite{Frigeri10}.

Because the Laplacian operator is compact, self-adjoint, and strictly positive with Robin boundary conditions in (\ref{Robin-boundary}), the Laplacian admits a countable collection of eigenfunctions and a corresponding non-decreasing set of positive eigenvalues. 
Also, one is able to define fractional powers of the Laplacian. 
Hence, the existence and uniqueness of a (local) weak solution can be sought through a Faedo-Galerkin approximation procedure by projecting the problem onto a subspace spanned by a finite set of eigenfunctions. 
Through the use of {\em{a priori}} estimates and compactness methods, the weak solution is usually obtained as a subsequence of a weakly converging sequence (see \cite{Lions69,Lions&Magenes72,Temam88} for a more detailed description of Faedo-Galerkin methods and their application to nonlinear partial differential equations). 
However, because of the nature of Problem (R) (in fact weak solutions are mild solutions), semigroups methods are utilized to obtain a local weak solution to Problem (R). 
Well-defined fractional powers of the Laplacian are usually utilized to decompose the solution operator into two operators: a part that decays to zero and a compact part. 
In either problem though, the growth assumptions on the nonlinear term $f(u)$ are not sufficient to apply such a decomposition. 
Thus, the existence of the global attractor is obtained through generalized semiflow methods contributed by Ball in \cite{Ball00,Ball04}. 
One must show that the solution operators are weakly continuous and asymptotically compact. 
On the other hand, not even fractional powers of the Laplacian are well-defined with acoustic boundary conditions. 
This means the solutions of Problem (A) cannot be obtained via a spectral basis, so local weak solutions to Problem (A) are obtained with semigroup methods. 
Both problems will be formulated in an abstract form and posed as an equation in a Banach space, containing a linear unbounded operator, which is the infinitesimal generator of a strongly continuous semigroup of contractions on the Banach space, and containing a locally Lipschitz nonlinear part. 
The global attractors for Problem (A) are also obtained through the weak continuity and asymptotic compactness of the the solution operators (again, cf. \cite{Ball00,Ball04}). 
Actually, the $\ep=1$ case of Problem (A) has already been studied in \cite{Frigeri10}, and it is that work, along with \cite{Beale&Rosencrans74}, that has brought the current one into view. 

The stability of solutions to partial differential equations under singular perturbations has been a topic undergoing rapid and strong growth; in particular, the continuity of attracting sets is a topic that has grown significantly. 
Only some of the results are mentioned below. 
An upper-semicontinuous family of global attractors for wave equations obtained from a perturbation of hyperbolic-relaxation type appears in \cite{Hale&Raugel88}. 
The problem is of the type
\[
\ep u_{tt}+u_t-\Delta u+\phi(u)=0,
\]
where $\ep\in [0,1]$. 
The equation possesses Dirichlet boundary conditions, and $\phi\in C^2(\mathbb{R})$ satisfies the growth assumption,
\[
\phi''(s)\leq C(1+|s|)
\]
for some $C>0$. 
The global attractor for the parabolic problem, $\mathcal{A}_0\subset H^2(\Omega)\cap H^1_0(\Omega)$, is ``lifted'' into the phase space for the hyperbolic problems, $X=H^1_0(\Omega)\times L^2(\Omega)$, by defining,
\begin{equation}  \label{HR-lift}
\mathcal{LA}_0:=\{(u,v)\in X:u\in\mathcal{A}_0,~v=f-g(u)+\Delta u\}.
\end{equation}
The family of sets in $X$ is defined by,
\begin{equation}  \label{gfam-1}
\mathbb{A}_\ep:=\left\{ \begin{array}{ll} \mathcal{L}\mathcal{A}_0 & \text{for}~\ep=0 \\ \mathcal{A}_\ep & \text{for}~\ep\in(0,1], \end{array}\right.
\end{equation}
where $\mathcal{A}_\ep\subset X$ denotes the global attractors for the hyperbolic-relaxation problem. 
The main result in \cite{Hale&Raugel88} is the upper-semicontinuity of the family of sets $\mathbb{A}_\ep$ in $X$; i.e.,
\begin{equation}  \label{4usc}
\lim_{\ep\rightarrow 0}\dist_X(\mathbb{A}_\ep,\mathbb{A}_0):= \lim_{\ep\rightarrow 0}\sup_{a\in\mathbb{A}_\ep}\inf_{b\in\mathbb{A}_0}\|a-b\|_X=0.
\end{equation}
To obtain this result, we will replace the initial conditions (\ref{acoustic-initial-conditions}) with the following 
\begin{equation}  \label{acoustic-initial-conditions2}
\ep\delta(0,\cdot) = \ep\delta_0, \quad \delta_t(0,\cdot) = \ep\delta_1-(1-\ep)u_0 \quad \text{on} \quad \{0\}\times \Gamma.
\end{equation}
{\em{The result in (\ref{4usc}) insures every that for every Robin-type problem, Problem (R), there is an acoustic relaxation, Problem (A), in which (\ref{4usc}) holds.}}

\begin{remark}  \label{key}
To motivate the choice of initial conditions chosen in (\ref{acoustic-initial-conditions2}), consider the following. The formal limit of (\ref{acoustic-boundary}) obtained in (\ref{key-2}) leads to the equation on $\Gamma$,
\[
\partial_{\mathbf{n}}u(t) = -u(t) + \phi,
\]
where $\phi=\phi(x)$ is an arbitrary function obtained from integration with respect to $t$. So at $t=0$ we find,
\[
\delta_t(0)=-u(0)+\phi.
\]
Thus, when (\ref{acoustic-initial-conditions2}) holds, we are guaranteed $\phi\equiv0.$ Moreover, this means the formal restriction ``$\ep=0$'' of Problem (A) does in fact coincide with Problem (R). 
\end{remark}

Since this result appeared, an upper-continuous family of global attractors for the Cahn-Hilliard equations has been found \cite{Zheng&Milani05}. 
Robust families of exponential attractors (that is, both upper- and lower-semicontinuous with explicit control over semidistances in terms of the perturbation parameter) of the type reported in \cite{GGMP05} have successfully been demonstrated to exist in numerous applications spanning partial differential equations of evolution: the Cahn-Hilliard equations with a hyperbolic-relaxation perturbation \cite{GGMP05-CH3D,GGMP05-CH1D}, applications with a perturbation appearing in a memory kernel have been treated for reaction diffusion equations, Cahn-Hilliard equations, phase-field equations, wave equations, beam equations, and numerous others \cite{GMPZ10}. 
Recently, the existence of an upper-semicontinuous family of global attractors for a reaction-diffusion equation with a singular perturbation of hyperbolic relaxation type and dynamic boundary conditions has appeared in \cite{Gal&Shomberg15}.
Robust families of exponential attractors have also been constructed for equations where the perturbation parameter appears in the boundary conditions. 
Many of these applications are to the Cahn-Hilliard equations and to phase-field equations \cite{Gal08,GGM08-2,Miranville&Zelik02}.
Also, continuous families of inertial manifolds have been constructed for wave equations \cite{Mora&Morales89-2}, Cahn-Hilliard equations \cite{BGM10}, and more recently, for phase-field equations \cite{Bonfoh11}. 
Finally, for generalized semiflows and for trajectory dynamical systems (dynamical systems where well-possedness of the PDE---uniqueness of the solution, in particular---is not guaranteed), some continuity properties of global attractors have been found for the Navier-Stokes equations \cite{Ball00}, the Cahn-Hilliard equations \cite{Segatti06}, and for wave equations \cite{Ball04,Zelik04}. 

The thrust behind robustness is typically an estimate of the form, 
\begin{equation}  \label{robust-intro}
\|S_\varepsilon(t)x-\mathcal{L}S_0(t)\Pi x\|_{X_\varepsilon}\leq C\varepsilon,
\end{equation}
where $x\in X_\varepsilon$, $S_\varepsilon(t)$ and $S_0(t)$ are semigroups generated by the solutions of the perturbed problem and the limit problem, respectively, $\Pi$ denotes a projection from $X_\varepsilon$ onto $X_0$ and $\mathcal{L}$ is a ``lift'' (such as (\ref{HR-lift})) from $X_0$ into $X_\varepsilon$.
The estimate (\ref{robust-intro}) means we can approximate the limit problem with the perturbation with control explicitly written in terms of the perturbation parameter. 
Usually, such control is only exhibited on compact time intervals. 
It is important to realize that the lift associated with a hyperbolic-relaxation problem, for example, requires a certain degree of regularity from the limit problem. 
In particular, \cite{Gal&Shomberg15,Hale&Raugel88} rely on (\ref{HR-lift}); so one needs $\mathcal{A}_0\subset H^2$ in order for $\mathcal{LA}_0\subset L^2$ to be well-defined.
For the above model problem, the perturbation parameter appears in a (dynamic) boundary condition. 
The perturbation is singular in nature, however, additional regularity from the global attractor $\mathcal{A}_0$ is not required in order for the lift to be well-defined. 
In place of this, we will rely on the (local) Lipschitz continuity of the corresponding semiflows; both the limit problem ($\ep=0$) and the perturbation problem ($\ep>0$) exhibit Lipschitz continuous solution operators.
This property allows us to prove an estimate of the form (\ref{robust-intro}), and additionally, the upper-semicontinuity of the family of global attractors, without additional regularity from $\mathcal{A}_0$. 
This fact is the key reason behind the motivation for using the first-order growth condition (\ref{f-assumption-1}); no further regularity is need to obtain an upper-continuous family of global attractors, and hence, a more general result is obtained. 

Unlike global attractors described above, exponential attractors (sometimes called, inertial sets) are positively invariant sets possessing finite fractal dimension that attract bounded subsets of the phase space exponentially fast. 
It can readily be seen that when both a global attractor $\mathcal{A}$ and an exponential attractor $\mathcal{M}$ exist, then $\mathcal{A}\subseteq \mathcal{M}$, and so the global attractor is also finite dimensional. 
When we turn our attention to proving the existence of exponential attractors, certain higher-order dissipative estimates are required. 
In the case for Problem (A), the estimates cannot be obtained along the lines of multiplication by fractional powers of the Laplacian; as we have already described, we need to resort to other methods.
In particular, we will apply $H^2$-elliptic regularity methods as in \cite{Pata&Zelik06}. 
Here, the main idea is to differentiate the equations with respect to time $t$ to obtain uniform estimates for the new equations. 
This strategy has recently received a lot of attention.
Some successes include dealing with a damped wave equation with acoustic boundary conditions \cite{Frigeri10} and a wave equation with a nonlinear dynamic boundary condition \cite{CEL02,CEL04-2,CEL04}.
Also, there is the hyperbolic relaxation of a Cahn-Hilliard equation with dynamic boundary conditions \cite{CGG11,Gal&Grasselli12}.
Additionally, this approach was also taken in \cite{Gal&Shomberg15}.
The drawback from using this approach comes in to difficulty in finding appropriate estimates that are {\em{uniform}} in the perturbation parameter $\ep.$
Indeed, this was the case in \cite{Gal&Shomberg15}.
There, the authors we able to find an upper-semicontinuous family of global attractors and a family of exponential attractors. 
It turned out that a certain higher-order dissipative estimate depends on $\ep$ in a crucial way, and consequently, the robustness/H\"older continuity of the family of exponential attractors cannot (yet) be obtained.
Furthermore, as it turns out, the global attractors found in \cite{Gal&Shomberg15} have finite (fractal) dimension, although the dimension is not necessarily independent of $\ep.$ 
It appears that similar difficulties persist with the model problem examined here.

The main results in this paper are:
\begin{itemize}
\item An upper-semicontinuity result for a {\em{generic}} family of sets for a family of semiflows, where in particular, the limit ($\ep=0$) semigroup of solution operators is locally Lipschitz continuous, uniformly in time on compact intervals. 
This result largely rests on the fact that the difference between two trajectories emanating from the same initial data, on compact time intervals, can be estimated in the phase space by a constant times $\sqrt{\ep}$.

\item Problem (R) and Problem (A) admit a family of global attractors $\{\mathcal{A}_\ep\}_{\ep\in[0,1]}$, where each is bounded, uniformly in $\ep$, in the respective phase space.
The rate of attraction of bounded sets to $\mathcal{A}_\ep$ approaches a constant as $\ep$ approaches zero. 

\item The generic semicontinuity result is applied to the family of global attractors $\{\mathcal{A}_\ep\}_{\ep\in[0,1]}$. 
Because of the Lipschitz continuity of the semiflow generated by Problem (R), $S_0$, no further regularity is required from the attractor $\mathcal{A}_0$ to insure $\mathcal{LA}_0$ is well-defined in the phase space of Problem (A).

\item Under more restrictive assumptions on the nonlinear terms $f$ and $g$, the global attractors are shown to possess optimal regularity by being bounded in a compact subset of the phase space; however, this bound is {\em{no longer}} independent of $\ep$.

\item There exists a family of exponential attractors $\{\mathcal{M}_\ep\}_{\ep\in[0,1]}$, admitted by the semiflows associated with for Problem (R) and Problem (A).
Since $\mathcal{A}_\ep\subset\mathcal{M}_\ep$ for each $\ep\in[0,1]$, this result insures the global attractors inherit  finite (fractal) dimension.
However, we cannot conclude that dimension is uniform in $\ep$ (this result remains open).
\end{itemize}

%%%%%%%%%%%%%%%%%%%%%%%%%%%%%%%%%%%%%%%%%%%%%%
\subsection{Notation and conventions}
%%%%%%%%%%%%%%%%%%%%%%%%%%%%%%%%%%%%%%%%%%%%%%

We take the opportunity here to introduce some notations and conventions that are used throughout the paper. 
We denote by $\Vert \cdot \Vert $, $\Vert \cdot \Vert _{k}$, the norms in $L^{2}(\Omega )$, $H^{k}(\Omega )$, respectively. 
We use the notation $\langle \cdot ,\cdot \rangle $ and $\langle \cdot ,\cdot \rangle_{k}$ to denote the products on $L^{2}(\Omega )$ and $H^{k}(\Omega)$, respectively.
For the boundary terms, $\Vert \cdot \Vert _{L^{2}(\Gamma )}$ and $\langle
\cdot ,\cdot \rangle _{L^{2}(\Gamma )}$ denote the norm and, respectively,
product on $L^{2}(\Gamma )$. 
We will require the norm in $H^{k}(\Gamma )$, to be denoted by $\Vert \cdot \Vert _{H^{k}(\Gamma )}$, where $k\geq 1$. 
The $L^{p}(\Omega )$ norm, $p\in (0,\infty ]$, is denoted $|\cdot |_{p}$. 
The dual pairing between $H^{1}(\Omega )$ and the dual $H^{-1}(\Omega) := (H^{1}(\Omega ))^{\ast }$ is denoted by $(u,v)_{H^{-1}\times H^1}$. 
In many calculations, functional notation indicating dependence on the variable $t$ is dropped; for example, we will write $u$ in place of $u(t)$. 
Throughout the paper, $C>0$ will denote a \emph{generic} constant, while $Q:\mathbb{R}_{+}\rightarrow \mathbb{R}_{+}$ will denote a \emph{generic} increasing function. 
All these quantities, unless explicitly stated, are \emph{independent} of the perturbation parameter $\varepsilon$.
Further dependencies of these quantities will be specified on occurrence.
We will use $\|B\|_{W}:=\sup_{\Upsilon\in B}\|\Upsilon\|_W$ to denote the ``size'' of the subset $B$ in the Banach space $W$.
Later in the article, we will rely on the Laplace-Beltrami operator $-\Delta_\Gamma$ on the surface $\Gamma.$ This operator is positive definite and self-adjoint on $L^2(\Gamma)$ with domain $D(-\Delta_\Gamma)$.
The Sobolev spaces $H^s(\Gamma)$, for $s\in\mathbb{R}$, may be defined as $H^s(\Gamma)=D((-\Delta_\Gamma)^{s/2})$ when endowed with the norm whose square is given by, for all $u\in H^s(\Gamma)$,
\begin{equation}  \label{LB-norm}
\|u\|^2_{H^s(\Gamma)} := \|u\|^2_{L^2(\Gamma)} + \left\|(-\Delta_\Gamma)^{s/2}u\right\|^2_{L^2(\Gamma)}.
\end{equation}

As for the plan of the paper, in Section 2 we review the important results concerning the limit ($\ep=0$) Problem (R), and in Section 3 we discuss the relevant results concerning the perturbation Problem (A). 
Some important remarks describing several instances of how Problem (A) depends of the perturbation parameter $\ep>0$ are given throughout Section 3. 
The Section 4 contains a new abstract upper-semicontinuity result that is then tailored specifically for the model problem under consideration. 
The final Section 5 summarizes our findings.

%%%%%%%%%%%%%%%%%%%%%%%%%%%%%%%%%%%%%%%%%%%%%%
%%%%%%%%%%%%%%%%%%%%%%%%%%%%%%%%%%%%%%%%%%%%%%
\section{Attractors for Problem (R), the $\ep=0$ case}  \label{s:Robin}
%%%%%%%%%%%%%%%%%%%%%%%%%%%%%%%%%%%%%%%%%%%%%%
%%%%%%%%%%%%%%%%%%%%%%%%%%%%%%%%%%%%%%%%%%%%%%

It is shown that Problem (R) possesses unique global weak solutions in a suitable phase space, and the solutions depend continuously on the initial data. 
Under suitable assumptions on $f$, we also establish the existence of global strong solutions.
The solutions generate a Lipschitz continuous semiflow which admits a bounded, absorbing, positively invariant set in the phase space. 
The existence of a global attractor is achieved by using Ball's asymptotic compactness method when we set $g\equiv 0$, and, with suitable assumptions on $f$, the existence of an exponential attractor follows from the recent work \cite{Gal&Shomberg15}.

%%%%%%%%%%%%%%%%%%%%%%%%%%%%%%%%%%%%%%%%%%%%%%
\subsection{The functional framework for Problem (R)}
%%%%%%%%%%%%%%%%%%%%%%%%%%%%%%%%%%%%%%%%%%%%%%

The phase space and the abstract Cauchy problem are formulated. 
In addition, the Laplacian, $-\Delta$, with the Robin boundary conditions described by (\ref{Robin-boundary}) is briefly discussed. 
The spectral basis is used in finding a Poincar\'e type inequality for Problem (R). 

Define
\[
\mathcal{H}_0:=H^1(\Omega) \times L^2(\Omega).
\]
The space $\mathcal{H}_0$ is Hilbert when endowed with the norm whose square is given by, for $\varphi=(u,v)\in\mathcal{H}_0$,
\begin{align*}
\|\varphi\|^2_{\mathcal{H}_0} &:= \|u\|^2_1 + \|u\|^2_{L^2(\Gamma)} + \|v\|^2 \\
& = \left( \|\nabla u\|^2 + \|u\|^2 \right) + \|u\|^2_{L^2(\Gamma)} + \|v\|^2.
\end{align*}
Recall that if $u\in H^1(\Omega)$, then $u_{\mid\Gamma} \in H^{1/2}(\Gamma)\hookrightarrow L^2(\Gamma)$ by the trace theorem, so $\|\cdot\|_{\mathcal{H}_0}$ is well-defined.
The embedding in the previous assertion follows from the embedding theorem on the {\em{two}}-dimensional surface $\Gamma$ (cf. \cite[Theorem 2.6]{Hebey99}).

As introduced in \cite{CEL02,Wu&Zheng06}, let $\Delta_{\mathrm{R}}:L^{2}(\Omega )\rightarrow L^{2}(\Omega )$ be the ``Robin-Laplacian'' operator with domain 
\begin{equation*}
D(\Delta_{\mathrm{R}}) = \{u\in H^2(\Omega ): \partial _{\bf{n}}u + u = 0 \ \text{on} \ \Gamma \}.
\end{equation*}
The Robin-Laplace operator $-\Delta_{\mathrm{R}}$ is self-adjoint and strictly positive; indeed, for all $u,v\in H^1(\Omega)$,
\[
(-\Delta_{\rm{R}} u,v)_{H^{-1}\times H^1} = (u,-\Delta_{\rm{R}} v)_{H^{-1}\times H^1},
\]
and 
\[
(-\Delta_{\rm{R}} u,u)_{H^{-1}\times H^1} = \|\nabla u\|^2 + \|u\|^2_{L^2(\Gamma)} \geq 0.
\]
By the spectral theorem, the operator $-\Delta_{\mathrm{R}}$ admits a system of eigenfunctions $(\omega_j)_{j=1}^\infty$ that is a complete orthonormal system in $L^2(\Omega)$, and a corresponding sequence of eigenvalues $(\lambda_j)_{j=1}^\infty$ which can be ordered into a nondecreasing sequence, such that $\lim_{j\rightarrow \infty}\lambda_j=\infty$. 
Using the Fourier series representation of $u$ the Poincar\'e inequality is computed as, for all $u\in H^1(\Omega)$,
\begin{equation}  \label{our-Poincare-0}
\|u\|\leq \frac{1}{\sqrt{\lambda_1}}\left( \|\nabla u\|^2 + \|u\|^2_{L^2(\Gamma)} \right)^{1/2},
\end{equation}
where $\lambda_1>0$ is the first eigenvalue of the Laplacian with Robin boundary conditions (\ref{Robin-boundary}). 
The Robin-Laplacian is extended to a continuous
operator $\Delta_{\mathrm{R}}:H^{1}(\Omega )\rightarrow \left( H^{1}(\Omega)\right) ^{\ast }$, defined by, for all $v\in H^{1}(\Omega )$, 
\begin{equation*}
(-\Delta_{\mathrm{R}}u,v)_{H^{-1}\times H^1}=\langle \nabla u,\nabla v\rangle +\langle u,v\rangle _{L^{2}(\Gamma )}.
\end{equation*}
We mention it is well-known that the Dirichlet trace map $\trD:C^\infty({\overline{\Omega}})\rightarrow C^\infty(\Gamma)$, defined by $\trD(u):= u_{\mid\Gamma}$, extends to a linear continuous operator $\trD:H^r(\Omega)\rightarrow H^{r-1/2}(\Gamma)$, for all $r>1/2$.
Hence, 
\[
u\in H^2(\Omega) \Longrightarrow u\in H^{3/2}(\Gamma) \quad\text{and}\quad\partial_{\bf{n}} u := \nabla u\cdot {\bf{n}} \in H^{1/2}(\Gamma).
\]
Thus, the equation $\partial_{\bf{n}}u= -u$ makes sense in the following distributional sense on $\Gamma$: for all $u_1\in D(\Delta_{\mathrm{R}})$ and $u_2\in H^1(\Omega)$,
\[
\int_\Omega (-\Delta_{\mathrm{R}})u_1 u_2 \diff x = \int_\Gamma u_1 u_2 \diff\sigma + \int_\Omega \nabla u_1 \nabla u_2\diff x.
\]

With $D(\Delta_{\rm{R}}),$ define the set
\begin{align*}
\mathcal{D}_0 & := D(\Delta_{\rm{R}}) \times H^1(\Omega) \\
& = \left\{ (u,v)\in H^2(\Omega) \times H^1(\Omega) : \partial_{\bf{n}}u + u = 0 \ \text{on} \ \Gamma \right\},
\end{align*}
and the linear unbounded operator $\mathrm{R}:D(\mathrm{R})\subset\mathcal{H}_0\rightarrow\mathcal{H}_0$, where $D(\mathrm{R})=\mathcal{D}_0$, by
\[
\mathrm{R}:=\begin{pmatrix} 0 & 1 \\ \Delta_{\rm{R}}-1 & -1 \end{pmatrix}.
\]
By the Lumer-Phillips theorem (cf. e.g. \cite[Theorem I.4.3]{Pazy83}) it is not hard to see that the operator $(\mathrm{R},D(\mathrm{R}))$ is an infinitesimal generator of a strongly continuous semigroup of contractions on $\mathcal{H}_0$, denoted $e^{\mathrm{R}t}$. The set $D(\mathrm{R})$ is dense in $\mathcal{H}_0=H^1(\Omega)\times L^2(\Omega)$, and $\mathrm{R}$ is dissipative since, for all $\varphi=(u,v)\in D(\mathrm{R})$,
\[
\langle \mathrm{R}\varphi,\varphi \rangle_{\mathcal{H}_0} = \langle v,u \rangle_1 + \langle v,u \rangle_{L^2(\Gamma)} + \langle \Delta_{\rm{R}} u-u-v,v \rangle = -\|v\|^2 \leq 0.
\]
The surjectivity requirement $(I+R)\varphi=\theta$ can be shown with the aid of the Lax-Milgram theorem; the elliptic system
\[
\left\{ \begin{array}{ll} u + v = \chi, & u\in D(\Delta_{\rm{R}}), \ v,\chi\in H^1(\Omega) \\ -\Delta_{\rm{R}} u + u = -\psi, & \psi \in L^2(\Omega). \end{array} \right.
\]
admits a unique weak solution $\varphi=(u,v)\in D(\mathrm{R})$ for any $\theta=(\chi,\psi)\in\mathcal{H}_0$. Define the map $\mathcal{F}:\mathcal{H}_0\rightarrow\mathcal{H}_0$ by 
\[
\mathcal{F}(\varphi):=\begin{pmatrix} 0 \\ -f(u) \end{pmatrix}
\]
for all $\varphi\in\mathcal{H}_0$. Since $f:H^1(\Omega)\rightarrow L^2(\Omega)$ is locally Lipschitz continuous \cite[cf. e.g. Theorem 2.7.13]{Zheng04}, it follows that the map $\mathcal{F}:\mathcal{H}_0\rightarrow\mathcal{H}_0$ is as well. Problem (R) may be put into the abstract form in $\mathcal{H}_0$
\begin{equation}  \label{abstract-Robin-problem}
\left\{
\begin{array}{l} \displaystyle\frac{\diff\varphi}{\diff t} = \mathrm{R}\varphi + \mathcal{F}(\varphi) \\ 
\varphi(0)=\varphi_0 \end{array}
\right.
\end{equation}
where $\varphi=\varphi(t)=(u(t),u_t(t))$ and $\varphi_0=(u_0,u_1)\in\mathcal{H}_0$; $v=u_t$ in the sense of distributions.

\begin{lemma}  \label{adjoint-r}
The adjoint of $\mathrm{R}$, denoted $\mathrm{R}^{\ast}$, is given by 
\begin{equation*}
\mathrm{R}^{\ast}:= -\begin{pmatrix} 0 & 1 \\ 
\Delta _{\mathrm{R}}-1 & 1 \end{pmatrix},
\end{equation*}
with domain 
\begin{equation*}
D(\mathrm{R}^{\ast }):=\{(\chi ,\psi )\in H^{2}(\Omega )\times
H^{1}(\Omega ):\partial _{\mathbf{n}}\chi +\chi =0 \ \text{on} \ \Gamma \}.
\end{equation*}
\end{lemma}

\begin{proof}
The proof is a calculation similar to, e.g., \cite[Lemma 3.1]{Ball04}.
\end{proof}

%%%%%%%%%%%%%%%%%%%%%%%%%%%%%%%%%%%%%%%%%%%%%%
\subsection{Well-posedness of Problem (R)}
%%%%%%%%%%%%%%%%%%%%%%%%%%%%%%%%%%%%%%%%%%%%%%

Semigroup methods are applied to obtain local weak/mild solutions. 
{\em{A priori}} estimates are then made to show that the local solutions are indeed global ones. 
Continuous dependence of the solution on the initial conditions and the uniqueness of solutions follows in the usual way by estimating the difference of two solutions. 
A semiflow (semigroup of solution operators) acting on the phase space is defined. Because of the continuous dependence estimate, the semiflow is Lipschitz continuous on the phase space. 
Due to the restrictive growth conditions imposed on $f(u)$, there is no regularity result. 

Formal multiplication of the PDE (\ref{damped-wave-equation}) by $2u_t$ in $L^2(\Omega)$ produces the {\em{energy equation}}
\begin{equation}\label{Robin-energy-3}
\frac{\diff}{\diff t} \left\{ \|\varphi\|^2_{\mathcal{H}_0} + 2\int_\Omega F(u) \diff x \right\} + 2\|u_t\|^2 = 0.
\end{equation}
Here, $F(s)=\int_0^s f(\sigma) \diff \sigma$. 

Note that from assumption (\ref{f-assumption-2}), it follows that there is a constant $\mu_0\in(0,1]$ such that, for all $\xi\in H^1(\Omega)$,
\begin{equation}\label{from-f-assumption-2}
2\int_\Omega F(\xi) \diff x \geq -(1-\mu_0)\|\xi\|^2_1 - \kappa_f
\end{equation}
for some constant $\kappa_f \ge 0$. 
A proof of (\ref{from-f-assumption-2}) can be found in \cite[page 1913]{CEL02}.
On the other hand, using (\ref{f-reg-ass-3}) and integration by parts on $F(s)=\int_{0}^{s}f(\sigma )\diff\sigma $, we have the upper-bound 
\begin{align}  \label{consequence-F-2}
\int_\Omega F(\xi)\diff x & \leq \langle f(\xi),\xi \rangle + \frac{\ell_2}{2\lambda_1}\|\xi\|^2_1.
\end{align}
Moreover, the inequality
\begin{equation}\label{from-f-assumption-1}
\langle f(u),u \rangle \geq -(1-\mu_0)\|u\|^2_1 - \kappa_f
\end{equation}
follows from the sign condition (\ref{f-assumption-2}) where $\mu_0\in(0,1]$ and $\kappa_f\geq0$ are from (\ref{from-f-assumption-2}). 

The definition of weak solution is from \cite{Ball77}.

\begin{definition} 
Let $T>0$. A map $\varphi\in C([0,T];\mathcal{H}_0)$ is a {\em{weak solution}} of (\ref{abstract-Robin-problem}) on $[0,T]$ if for each $\theta\in D(\mathrm{R}^*)$ the map $t \mapsto \langle \varphi(t),\theta \rangle_{\mathcal{H}_0}$ is absolutely continuous on $[0,T]$ and satisfies, for almost all $t\in[0,T]$,
\begin{equation}  \label{abs-1}
\frac{\diff}{\diff t}\langle \varphi(t),\theta \rangle_{\mathcal{H}_0} = \langle \varphi(t),\mathrm{R}^*\theta \rangle_{\mathcal{H}_0} + \langle \mathcal{F}(\varphi(t)),\theta \rangle_{\mathcal{H}_0}.
\end{equation}
The map $\varphi$ is a weak solution on $[0,\infty)$ (i.e. a {\em{global weak solution}}) if it is a weak solution on $[0,T]$ for all $T>0$. 
\end{definition}

According to \cite[Definition 3.1 and Proposition 3.5]{Ball04}, the notion of weak solution above is equivalent to the following notion of a mild solution.

\begin{definition}
Let $T>0$. A function $\varphi:[0,T]\rightarrow\mathcal{H}_0$ is a weak/mild solution of (\ref{abstract-Robin-problem}) on $[0,T]$ if and only if $\mathcal{F}(\varphi(\cdot))\in L^1(0,T;\mathcal{H}_0)$ and $\varphi$ satisfies the variation of constants formula, for all $t\in[0,T],$
\[
\varphi(t)=e^{\mathrm{R}t}\varphi_0 + \int_0^t e^{\mathrm{R}(t-s)}\mathcal{F}(\varphi(s)) \diff s.
\]
\end{definition}

Furthermore, by \cite[Proposition 3.4]{Ball04} and the explicit characterization of $D(\mathrm{R}^*)$, our notion of weak solution is also equivalent to the standard concept of a weak (distributional) solution to Problem (R).

\begin{definition}  \label{weak}
A function $\varphi=(u,u_{t}):[0,T]\rightarrow \mathcal{H}_0$ is a weak solution of (\ref{abstract-Robin-problem}) (and, thus of (\ref{damped-wave-equation})-(\ref{Robin-boundary})) on $[0,T],$ if, for almost all $t\in \left[ 0,T\right],$
\begin{equation*}
\varphi =(u,u_{t})\in C(\left[ 0,T\right] ;\mathcal{H}_0),
\end{equation*}
and, for each $\psi \in H^{1}\left( \Omega \right) ,$ $\langle u_{t},\psi \rangle \in C^{1}\left( \left[ 0,T\right] \right) $ with
\begin{equation*}
\frac{\diff}{\diff t}\left\langle u_{t}\left( t\right) ,\psi \right\rangle + \left\langle u_{t}\left( t\right) ,\psi \right\rangle + \left\langle u\left( t\right) ,\psi \right\rangle_1 + \left\langle u\left( t\right) ,\psi \right\rangle _{L^{2}\left( \Gamma \right) }=-\left\langle f\left( u\left( t\right) \right) ,\psi
\right\rangle. 
\end{equation*}
\end{definition}

Indeed, by \cite[Lemma 3.3]{Ball04} we have that $f:H^{1}\left( \Omega
\right) \rightarrow L^{2}\left( \Omega \right) $ is sequentially weakly
continuous and continuous thanks to the assumptions (\ref{f-assumption-1}) and (\ref{f-assumption-2}). 
Moreover, by \cite[Proposition 3.4]{Ball04} and the explicit representation of $D(\mathrm{R}^*)$ in Lemma \ref{adjoint-r}, $\langle \varphi_{t},\theta \rangle \in C^{1}\left( \left[ 0,T\right] \right) $ for all $\theta \in D\left( \mathrm{R}^{\ast }\right) $, and (\ref{abs-1}) is satisfied. 

Finally, the notion of strong solution to Problem (R) is as follows.

\begin{definition}  \label{Robin-strong}
Let $\varphi_0 = (u_0,u_1) \in \mathcal{D}_0$: that is, let $\varphi_0\in H^{2}(\Omega )\times H^{1}(\Omega )$ be such that 
\begin{equation*}
\partial _{\bf{n}}u_{0}+u_{0}=0 \ \text{on} \ \Gamma
\end{equation*}
is satisfied. 
A function $\varphi(t) = (u(t),u_t(t))$ is called a (global) strong solution if it is a (global) weak solution in the sense of Definition \ref{weak} and if it satisfies the following regularity properties:
\begin{equation}  \label{regularity-property}
\varphi \in L^{\infty }([0,\infty) ;\mathcal{D}_0)\quad\text{and}\quad\partial_t\varphi\in L^{\infty }([0,\infty) ;\mathcal{H}_0).
\end{equation}
Therefore, $\varphi(t) = (u(t), u_t(t)) $ satisfies the equations (\ref{damped-wave-equation})-(\ref{Robin-boundary}) almost everywhere; i.e., is a strong solution.
\end{definition}

We now have the first main result. 

\begin{theorem}\label{t:Robin-weak-solutions}
Assume (\ref{f-assumption-1}) and (\ref{f-assumption-2}) hold.
Let $\varphi_0\in\mathcal{H}_0$. 
Then there exists a unique global weak solution $\varphi\in C([0,\infty);\mathcal{H}_0)$ to (\ref{abstract-Robin-problem}). 
For each weak solution, the map 
\begin{equation}\label{C1-map}
t \mapsto \|\varphi(t)\|^2_{\mathcal{H}_0} + 2\int_\Omega F(u(t)) \diff x
\end{equation}
is $C^1([0,\infty))$ and the energy equation (\ref{Robin-energy-3}) holds (in the sense of distributions). 
Moreover, for all $\varphi_0, \theta_0\in\mathcal{H}_0$, there exists a positive constant $\nu_0>0$, depending on $\|\varphi_0\|_{\mathcal{H}_0}$ and $\|\theta_0\|_{\mathcal{H}_0}$, such that, for all $t\geq 0$,
\begin{equation}\label{continuous-dependence}
\|\varphi(t)-\theta(t)\|_{\mathcal{H}_0} \leq e^{\nu_0 t} \|\varphi_0 - \theta_0\|_{\mathcal{H}_0}.
\end{equation}
Additionally, when (\ref{f-assumption-2}), (\ref{f-reg-ass-2}), and (\ref{f-reg-ass-3}) hold, and $\varphi_0\in \mathcal{D}_0$, there exists a unique global strong solution $\varphi\in C([0,\infty);\mathcal{D}_0)$ to (\ref{abstract-Robin-problem}). 
\end{theorem}

\begin{proof}
As discussed in the previous sections, the operator $\mathrm{R}$ with domain $D(\mathrm{R})=D(\Delta_{\rm{R}}) \times H^1(\Omega)$ is an infinitesimal generator of a strongly continuous semigroup of contractions on $\mathcal{H}_0$, and the map $\mathcal{F}:\mathcal{H}_0\rightarrow\mathcal{H}_0$ is locally Lipschitz continuous. Therefore, by, e.g. \cite[Theorem 2.5.4]{Zheng04}, for any $\varphi_0\in\mathcal{H}_0$, there is a $T^*>0$, depending on $\|\varphi_0\|_{\mathcal{H}_0}$, such that the abstract Problem (\ref{abstract-Robin-problem}) admits a unique local weak solution on $[0,T^*)$ satisfying 
\[
\varphi\in C([0,T^*);\mathcal{H}_0).
\]

The next step is to show that $T^*=\infty$. Since the map (\ref{C1-map}) is absolutely continuous on $[0,T^*)$, then integration of the energy equation (\ref{Robin-energy-3}) over $(0,t)$ yields, for all $t\in[0,T^*)$,
\begin{equation}\label{Robin-energy-4}
\|\varphi(t)\|^2_{\mathcal{H}_0} + 2\int_\Omega F(u(t)) \diff x + 2\int_0^t \|u_t(\tau)\|^2 d\tau= \|\varphi_0\|^2_{\mathcal{H}_0} + 2\int_\Omega F(u_0) \diff x.
\end{equation}
Applying inequality (\ref{from-f-assumption-2}) to (\ref{Robin-energy-4}), omitting the remaining (positive) integral on the left hand side of (\ref{Robin-energy-4}), and using the fact that the estimate $|F(s)|\leq C|s|(1+|s|^3)$ follows from assumption (\ref{f-assumption-1}), for some constant $C>0$, then, for all $t\in[0,T^*)$,
\begin{equation}\label{Robin-bound-1}
\|\varphi(t)\|_{\mathcal{H}_0} \leq C \|\varphi_0\|_{\mathcal{H}_0}.
\end{equation}
Since the bound on the right hand side of (\ref{Robin-bound-1}) is independent of $t\in[0,T^*)$, $T^*$ can be extended indefinitely; therefore $T^*=+\infty$.

To show that (\ref{continuous-dependence}) holds, let $\varphi_0=(u_0,u_1), \theta_0=(v_0,v_1)\in\mathcal{H}_0$. Let $\varphi(t)=(u(t),u_t(t))$ and, respectively, $\theta(t)=(v(t),v_t(t))$ denote the corresponding global solutions of Problem (R) on $[0,\infty)$ with the initial data $\varphi_0$ and $\theta_0$.
For all $t\ge0$, 
\begin{align}
\bar\varphi(t) & := \varphi(t)-\theta(t)  \notag \\
& = \left(( u(t),u_t(t) \right) - \left( v(t),v_t(t) \right)  \notag \\
& =: \left( z(t),z_t(t) \right),  \notag
\end{align}
and 
\begin{align}
\bar\varphi_0 & := \varphi_0-\theta_0  \notag \\ 
& = (u_0-v_0,u_1-v_1)  \notag \\ 
& =: (z_0,z_1).  \notag
\end{align}
Then $z$, satisfies the IBVP
\begin{equation}\label{dependence-1}
\left\{ \begin{array}{ll} z_{tt}+z_t-\Delta z+z+f(u)-f(v)=0 & \text{in} \ (0,\infty)\times\Omega \\ z(0,\cdot)=z_0, ~z_t(0,\cdot)=z_1 & \text{on} \ \{0\}\times\Omega \\ \partial_{\bf{n}}z+z=0 & \text{on} \ (0,\infty)\times\Gamma. \end{array} \right.
\end{equation}
Multiply (\ref{dependence-1})$_1$ by $2z_t$ in $L^2(\Omega)$ to yield, for almost all $t\geq0$,
\begin{equation}\label{dependence-2}
\frac{\diff}{\diff t} \|\bar\varphi(t)\|^2_{\mathcal{H}_0} + 2\|z_t\|^2 = 2\langle f(v)-f(u),z_t \rangle.
\end{equation}
Since $f:H^1(\Omega)\rightarrow L^2(\Omega)$ is locally Lipschitz continuous, then
\begin{equation}\label{dependence-3}
2|\langle f(v)-f(u),z_t \rangle| \leq C\|z\|^2_1 + \|z_t\|^2,
\end{equation}
where the constant $C>0$ depends on the local Lipschitz continuity of $f$ as well as the uniform bound on the weak solutions $u$ and $v$ thanks to (\ref{Robin-bound-1}). 
After omitting the term $2\|z_t\|^2$ on the left hand side of (\ref{dependence-2}) and adding $\|z\|^2_{L^2(\Gamma)}$ to the right hand side, combining (\ref{dependence-2}) and (\ref{dependence-3}) produces, for almost all $t\geq 0$,
\[
\frac{\diff}{\diff t}\|\bar\varphi(t)\|^2_{\mathcal{H}_0} \leq C\|\bar\varphi(t)\|^2_{\mathcal{H}_0}.
\]
The mapping $t\mapsto \|(z(t),z_t(t))\|^2_{\mathcal{H}_0}$ is absolutely continuous on $(0,\infty)$, so equation (\ref{continuous-dependence}) is obtained after applying a Gr\"onwall inequality. 
The uniqueness of the solutions now follows from (\ref{continuous-dependence}).

To prove the existence of strong solutions, let $\varphi_0\in\mathcal{D}_0$ and assume (\ref{f-assumption-2}), (\ref{f-reg-ass-2}), and (\ref{f-reg-ass-3}) hold.
(Recall, $\mathcal{D}_0=D(\mathrm{R})$.)
We know that $\mathcal{F}:D(R)\rightarrow D(\mathrm{R})$ is locally Lipschitz continuous thanks to (\ref{f-reg-ass-2}).
It now follows that (cf. e.g. \cite[Theorem 2.5.6]{Zheng04}) there exists a unique (global) solution; i.e., a weak solution satisfying 
\begin{equation*}
\varphi\in C^1([0,\infty);\mathcal{H}_0) \cap C^0([0,\infty);D(\mathrm{R})).
\end{equation*} 
Thus, (\ref{regularity-property}) holds.
This finishes the proof.
\end{proof}

\begin{remark}
The interested author may also view \cite{Wu&Zheng06} where strong solutions are obtained for (\ref{damped-wave-equation}) with the dynamic boundary condition (\ref{DBC}).
There, the authors also show the convergence of strong solutions to equilibrium when $f$ is (real) analytic.
\end{remark}

The following provides the dynamical system we associate with Problem (R).

\begin{corollary}  \label{sf-r}
Let $\varphi_0=(u_0,u_1)\in\mathcal{H}_0$ and $u$ be the unique global solution of Problem (R). 
The family of maps $S_0=(S_0(t))_{t\geq 0}$ defined by 
\[
S_0(t)\varphi_0(x):=(u(t,x,u_0,u_1),u_t(t,x,u_0,u_1))
\]
is a {\em{semiflow}} generated by Problem (R).
The operators $S_0(t)$ satisfy
\begin{enumerate}
	\item $S_0(t+s)=S_0(t)S_0(s)$ for all $t,s\geq 0$.
	\item $S_0(0)=I_{\mathcal{H}_0}$ (the identity on $\mathcal{H}_0$)
	\item $S_0(t)\varphi_0\rightarrow S_0(t_0)\varphi_0$ for every $\varphi_0\in\mathcal{H}_0$ when $t\rightarrow t_0$.
\end{enumerate}

Additionally, each mapping $S_0(t):\mathcal{H}_0\rightarrow\mathcal{H}_0$ is Lipschitz continuous, uniformly in $t$ on compact intervals; i.e., for all $\varphi_0, \theta_0\in\mathcal{H}_0$, and for each $T\geq 0$, and for all $t\in[0,T]$,
\begin{equation}\label{S0-Lipschitz-continuous}
\|S_0(t)\varphi_0-S_0(t)\theta_0\|_{\mathcal{H}_0} \leq e^{\nu_0 T}\|\varphi_0-\theta_0\|_{\mathcal{H}_0}.
\end{equation}
\end{corollary}

\begin{proof}
The semigroup properties (1) and (2) are well-known and apply to a general class of abstract Cauchy problems possessing many applications (see \cite{Babin&Vishik92,Goldstein85,Morante79,Tanabe79}; in particular, a proof of property (1) is given in \cite[\S1.2.4]{Milani&Koksch05}). 
The continuity in $t$ described by property (3) follows from the definition of weak solution (this also establishes strong continuity of the operators when $t_0=0$). 
The continuity property (\ref{S0-Lipschitz-continuous}) follows from (\ref{continuous-dependence}).
\end{proof}

%%%%%%%%%%%%%%%%%%%%%%%%%%%%%%%%%%%%%%%%%%%%%%
\subsection{Dissipativity of Problem (R)}
%%%%%%%%%%%%%%%%%%%%%%%%%%%%%%%%%%%%%%%%%%%%%%

We will now show that the dynamical system $(S_0(t),\mathcal{H}_0)$ generated by the weak solutions of Problem (R) is dissipative in the sense that $S_0$ admits a closed, positively invariant, bounded absorbing set in $\mathcal{H}_0$. 

The following functional will be of use in the proof of the following theorem.
For $\varphi=(u,v)\in\mathcal{H}_0$, define the map $E_0:\mathcal{H}_0\rightarrow\mathbb{R}$ by
\begin{equation}  \label{Robin-functional-0}
E_0(\varphi) :=\|\varphi\|^2_{\mathcal{H}_0} + 2\mu_0\langle u,v \rangle + 2\int_\Omega F(u) \diff x.
\end{equation}
Using (\ref{from-f-assumption-2}) and the growth estimate (\ref{f-assumption-1}), the functional $E_0(\varphi)$ satisfies, for some constants $C_1,C_2>0$ and for all $\varphi\in\mathcal{H}_0$,
\begin{equation}  \label{Robin-functional-1}
C_1\|\varphi\|^2_{\mathcal{H}_0} - \kappa_f \leq E_0(\varphi) \leq C_2\|\varphi\|_{\mathcal{H}_0}(1+\|\varphi\|^3_{\mathcal{H}_0}).
\end{equation}

\begin{lemma}
Assume (\ref{f-assumption-1}) and (\ref{f-assumption-2}) hold.
There exists $R_0>0$ with the property that: for every $R>0$, there exists $t_0=t_0(R)\geq 0$, depending on $R$, such that for every $\varphi_0\in\mathcal{H}_0$ with $\|\varphi_0\|_{\mathcal{H}_0}\le R$ and for every $t\geq t_0$,
\[
\|S_0(t)\varphi_0\|_{\mathcal{H}_0}\leq R_0.
\]
Furthermore, the set 
\begin{equation}  \label{abss-0}
\mathcal{B}_0:=\{ \varphi\in \mathcal{H}_0 : \|\varphi\|_{\mathcal{H}_0} \le R_0 \}
\end{equation}
is closed, bounded, absorbing, and positively invariant for the semiflow $S_0$ in $\mathcal{H}_0$.
\end{lemma}

\begin{proof}
Multiply equation (\ref{damped-wave-equation}) by $u_t+\eta u$, where $\eta>0$ is a sufficiently small constant yet to be chosen, and integrate over $\Omega$ to yield, for almost all $t\geq 0$,
\begin{equation}  \label{Robin-absorbing-set-1}
\frac{1}{2}\frac{\diff}{\diff t}E_0(\varphi) + \eta\|u\|^2_1 + \eta\|u\|^2_{L^2(\Gamma)} + \eta\langle u,u_t \rangle + (1-\eta)\|u_t\|^2 + \eta\langle f(u),u \rangle = 0.
\end{equation}
Directly from (\ref{from-f-assumption-1})
\begin{equation}\label{from-f-1-and-Poincare}
\eta\langle f(u),u \rangle \ge -\eta(1-\mu_0)\|u\|^2_1 - \eta\kappa_f.
\end{equation}
With the basic estimate,
\begin{equation}  \label{fuk-1}
\eta\langle u,u_t \rangle \ge -\frac{\eta\mu_0}{2}\|u\|^2_1 - \frac{\eta}{2\mu_0}\|u_t\|^2,
\end{equation}
and after also applying (\ref{from-f-1-and-Poincare}) to equation (\ref{Robin-absorbing-set-1}), we find that for each $\mu_0\in(0,1]$ and $0<\eta<\left( 1+\frac{1}{2\mu_0} \right)^{-1}$, there is a constant $m_0=m_0(\mu_0,\eta)>0$ in which, for almost all $t\geq 0$,
\[
\frac{\diff}{\diff t}E_0(\varphi) + 2m_0\|\varphi\|^2_{\mathcal{H}_0} \le 2\eta\kappa_f.
\]

Let $\widetilde R>0$. 
For all $\varphi_0\in \mathcal{H}_0$ with $\|\varphi_0\|_{\mathcal{H}_0}\leq \widetilde R$, the upper-bound in (\ref{Robin-functional-1}) reads
\[
E_0(\varphi_0) \leq C_2\widetilde R(1+\widetilde{R}^3).
\]
Hence, for all $\widetilde R>0$, there exists $R>0$ such that, for all $\varphi_0\in\mathcal{H}_0$ with $\|\varphi_0\|_{\mathcal{H}_0}\le \widetilde R$, then $E_0(\varphi_0)\le R.$
From the lower-bound in (\ref{Robin-functional-1}), we immediately see that $\sup_{t\ge0}E_0(\varphi(t))\ge - \kappa_f.$
By Lemma \ref{t:diff-ineq-1} there exists $t_0>0$, depending on $\kappa_f$, such that for all $t\geq t_0$ and for all $\varphi_0\in \mathcal{H}_0$ with $\|\varphi_0\|_{\mathcal{H}_0}\leq \widetilde R$,
\[
E_0(S_0(t)\varphi_0) \le \sup_{\varphi\in \mathcal{H}_0} \left\{ E_0(\varphi) : m_0\|\varphi\|^2_{\mathcal{H}_0}\leq \eta\kappa_f \right\}.
\]
Thus, there is $R_0>0$ such that, for all $t\geq t_0$ and for all $\varphi_0\in \mathcal{H}_0$ with $\|\varphi_0\|_{\mathcal{H}_0}\leq \widetilde R$,
\begin{equation}\label{Robin-semiflow-absorbing}
\|S_0(t)\varphi_0\|_{\mathcal{H}_0} \leq R_0.
\end{equation}
By definition, the set $\mathcal{B}_0$ in (\ref{abss-0}) is closed and bounded in $\mathcal{H}_0$. 
The inequality in (\ref{Robin-semiflow-absorbing}) implies that $\mathcal{B}_0$ is absorbing: given any nonempty bounded subset $B$ of $\mathcal{H}_0$, there is a $t_0\geq 0$ depending on $B$ in which, for all $t\geq t_0$, $S_0(t)B\subseteq \mathcal{B}_0$. 
Consequently, since $\mathcal{B}_0$ is bounded, $\mathcal{B}_0$ is also positively invariant under the semiflow $S_0$.
\end{proof}

\begin{remark}  \label{Robin-time}
According to Lemma \ref{t:diff-ineq-1}, $t_0$ depends on $\eta$, $m_0$ (both described above in the proof) and $\kappa_f$ as 
\[
t_0=\frac{1}{\iota}\left( C_2R(1+R^3)+\eta\kappa_f \right).
\] 
Moreover, thanks to the bounds given in (\ref{Robin-functional-1}), and with the differential inequality (\ref{id-007}), we may explicitly compute the radius of $\mathcal{B}_0$.
To begin, suppose $\varphi_0\in\mathcal{H}_0$ is such that $\|\varphi_0\|_{\mathcal{H}_0}\le R.$
Integrate (\ref{id-007}) on $[0,t]$.
After applying (\ref{Robin-functional-1}) twice, we obtain the inequality
\[
\|\varphi(t)\|_{\mathcal{H}_0}\le C_1^{1/2} \left( (\eta\kappa_f-m_0R^2)t + C_2R(1+R^3)+\eta\kappa_f \right)^{1/2}.
\]
Since (\ref{Robin-bound-1}) must hold, it must be the case that $\eta\kappa_f-m_0 R^2<0;$ i.e.,
\[
R>\sqrt{\frac{\eta\kappa_f}{m_0}}.
\]
Thus, we set $R=R(\iota)=\sqrt{\frac{\eta\kappa_f+\iota}{m_0}}$, for any $\iota>0$, and therefore obtain the radius of $\mathcal{B}_0$ to be,
\begin{equation}  \label{reg-radii}
R_0^2(\iota) := \frac{C_2}{C_1} \left(\frac{\eta\kappa_f+\iota}{m_0}\right)^{1/2} \left( \eta\kappa_f+1 + \left( \frac{\eta\kappa_f+\iota}{m_0} \right)^{3/2} \right).
\end{equation}
\end{remark}

%%%%%%%%%%%%%%%%%%%%%%%%%%%%%%%%%%%%%%%%%%%%%%
\subsection{Global attractor for Problem (R)}
%%%%%%%%%%%%%%%%%%%%%%%%%%%%%%%%%%%%%%%%%%%%%%

We now aim to prove

\begin{theorem}  \label{t:robin-global}
Assume (\ref{f-assumption-1}) and (\ref{f-assumption-2}) hold.
The semiflow $S_0$ generated by the weak solutions of Problem (R) admits a global attractor $\mathcal{A}_0$ in $\mathcal{H}_0$. The global attractor is invariant under the semiflow $S_0$ (both positively and negatively) and attracts all nonempty bounded subsets of $\mathcal{H}_0$; precisely, 
\begin{enumerate}
\item For each $t\geq 0$, $S_0(t)\mathcal{A}_0=\mathcal{A}_0$, and 
\item For every nonempty bounded subset $B$ of $\mathcal{H}_0$,
\[
\lim_{t\rightarrow\infty}{\rm{dist}}_{\mathcal{H}_0}(S_0(t)B,\mathcal{A}_0):=\lim_{t\rightarrow\infty}\sup_{\varphi\in B}\inf_{\theta\in\mathcal{A}_0}\|S_0(t)\varphi-\theta\|_{\mathcal{H}_0}=0.
\]
\end{enumerate}
The global attractor is unique and given by 
\[
\mathcal{A}_0=\omega(\mathcal{B}_0):=\bigcap_{s\geq 0}{\overline{\bigcup_{t\geq s} S_0(t)\mathcal{B}_0}}^{\mathcal{H}_0}.
\]
Furthermore, $\mathcal{A}_0$ is the maximal compact invariant subset in $\mathcal{H}_0$.
\end{theorem}

In order to prove Theorem \ref{t:robin-global}, we now develop further important properties of the semiflow $S_0$; the semiflow is weakly continuous and asymptotically compact. 
Both properties utilize only assumptions (\ref{f-assumption-1}) and (\ref{f-assumption-2}) on the nonlinear term $f(u)$. 
Consequently, the method of decomposing the semiflow into decay and compact parts to deduce the existence of the global attractor cannot be applied. 
A different approach is taken to deduce the existence of a global attractor; whereby the semiflow is shown to be weakly continuous and asymptotically compact (see \cite{Frigeri10}). 
Since the semiflow admits a bounded absorbing set, it follows from the theory of generalized semiflows by Ball (cf. \cite{Ball00,Ball04}) that the semiflow $S_0$ also admits a global attractor in the phase space $\mathcal{H}_0$.

The first property follows from the general result in \cite[Theorem 3.6]{Ball04}.

\begin{lemma}  \label{t:Robin-weakly-continuous}
The semiflow $S_0$ is weakly continuous on $\mathcal{H}_0$; i.e., for each $t\geq 0$,
\[
S_0(t)\varphi_{0n}\rightharpoonup S_0(t)\varphi_0 \ \text{in} \ \mathcal{H}_0 ~\text{when}~\varphi_{0n}\rightharpoonup \varphi_0 \ \text{in} \ \mathcal{H}_0.
\]
\end{lemma}

\begin{lemma}  \label{t:Robin-asymptotic-compactness}
The semiflow $S_0$ is asymptotically compact in $\mathcal{H}_0$; i.e., if $\varphi_{0n}=(u_{0n},u_{1n})$ is any bounded sequence in $\mathcal{H}_0$ and if $t_n$ is any sequence such that $t_n\rightarrow\infty$ as $n\rightarrow \infty$, then the sequence $\varphi^{(n)}(t_n) = S_0(t_n)\varphi_{0n}$ has a convergent subsequence.
\end{lemma}

\begin{proof}
The proof essentially follows from \cite[Proposition 2]{Frigeri10}. 
Indeed, the proof needed here is much simpler because of the (static) Robin boundary condition. 
\end{proof}

The existence of a global attractor for the dynamical system $(S_0(t),\mathcal{H}_0)$ now follows from \cite[Theorem 3.3]{Ball00}.

%%%%%%%%%%%%%%%%%%%%%%%%%%%%%%%%%%%%%%%%%%%%%%
\subsection{Optimal regularity for $\mathcal{A}_0$}
%%%%%%%%%%%%%%%%%%%%%%%%%%%%%%%%%%%%%%%%%%%%%%

In this section we report the result concerning the optimal regularity for the global attractor associated with Problem (R). 
It is at this point where we need to assume further smoothness from the nonlinear term $f$. 
Moving forward, we assume (\ref{f-assumption-2}) and (\ref{f-reg-ass-2})-(\ref{f-reg-ass-3}).
With these assumptions, we find the asymptotic compactness property for the weak solutions from \cite[Theorem 3.17]{Gal&Shomberg15}.

\begin{theorem}  \label{exp-attr-r-1} 
Assume (\ref{f-assumption-2}), and (\ref{f-reg-ass-2})-(\ref{f-reg-ass-3}) hold.
There exists a closed and bounded subset $\mathcal{U}_0\subset \mathcal{D}_0$, such that for every nonempty bounded subset $B\subset \mathcal{H}_0$,
\begin{equation}
\mathrm{dist}_{\mathcal{H}_0}(S_0(t)B,\mathcal{U}_0) \le Q(\left\Vert B\right\Vert _{\mathcal{H}_0}) e^{-\omega_0 t},  \label{tran-1}
\end{equation}
for some constant $\nu_0>0$.
\end{theorem}

By more standard arguments in the theory of attractors (see, e.g., \cite{Hale88,
Temam88}), it follows that the global attractor $\mathcal{A}_0 \subset \mathcal{U}_0$ for the semigroup $S_0(t)$ is bounded in $\mathcal{D}_0$.

\begin{corollary}
The global attractor admitted by $S_0$ for Problem (R) satisfies 
\begin{equation*}
\mathcal{A}_0\subset\mathcal{U}_0.
\end{equation*}
Consequently, the global attractor $\mathcal{A}_0$ is bounded in $\mathcal{D}_0$ and consists only of strong solutions. 
\end{corollary}

%%%%%%%%%%%%%%%%%%%%%%%%%%%%%%%%%%%%%%%%%%%%%%
\subsection{Exponential attractor for Problem (R)}
%%%%%%%%%%%%%%%%%%%%%%%%%%%%%%%%%%%%%%%%%%%%%%

The existence of an exponential attractor depends on certain properties of the semigroup; namely, the smoothing property for the difference of any two trajectories and the existence of a more regular bounded absorbing set in the phase space.

The existence of exponential attractors for Problem (R) follows directly from \cite{Gal&Shomberg15}. 
Indeed, Theorem \ref{ear} (below) was shown to hold for (\ref{damped-wave-equation}), under assumptions (\ref{f-assumption-2}) and (\ref{f-reg-ass-2})-(\ref{f-reg-ass-3}), with the dynamic boundary condition (\ref{DBC}).
Moreover, the existence of the bounded absorbing set $\mathcal{B}_{0}^1\subset \mathcal{D}_0$ for the limit semigroup $S_{0}(t)$ associated with Problem (R) was established in \cite[Lemma 4.6]{Gal&Shomberg15}.

\begin{theorem}  \label{ear} 
Assume (\ref{f-assumption-2}), and (\ref{f-reg-ass-2})-(\ref{f-reg-ass-3}) hold.
The dynamical system $\left( S_0,\mathcal{H}_0\right)$ associated with Problem (R) admits an exponential attractor $\mathcal{M}_0$ compact in $\mathcal{H}_0$, and bounded in $\mathcal{U}_0$. 
Moreover, there hold:

(i) For each $t\geq 0$, $S_0(t)\mathcal{M}_0\subseteq \mathcal{M}_0$.

(ii) The fractal dimension of $\mathcal{M}_0$ with respect to the metric $\mathcal{H}_0$ is finite, namely,
\begin{equation*}
\dim_{\mathrm{F}}\left( \mathcal{M}_0,\mathcal{H}_0\right) \leq C<\infty,
\end{equation*}
for some positive constant $C$.

(iii) There exist $\omega_1>0$ and a nonnegative monotonically increasing function $Q(\cdot)$ such that, for all $t\geq 0$, 
\begin{equation}  \label{exp-attn-9}
\dist_{\mathcal{H}_0}(S_0(t)B,\mathcal{M}_0)\leq Q(\Vert B\Vert _{\mathcal{H}_0})e^{-\omega_1 t},
\end{equation}
for every nonempty bounded subset $B$ of $\mathcal{H}_0$.
\end{theorem}

The proof of Theorem \ref{ear} follows from the
application of an abstract result tailored specifically to our needs (see, e.g., \cite[Proposition 1]{EMZ00}, \cite{FGMZ04}, \cite{GGMP05}; cf. also Remark \ref{rem_att}\ below).

\begin{remark}  
Above,
\begin{equation*}
\dim_{\mathrm{F}}(\mathcal{M}_0,\mathcal{H}_0):=\limsup_{r\rightarrow 0}\frac{\ln \mu _{\mathcal{H}_0}(\mathcal{M}_0,r)}{-\ln r}<\infty ,
\end{equation*}
where, $\mu _{\mathcal{H}_0}(\mathcal{X},r)$ denotes the minimum number of $r$-balls from $\mathcal{H}_0$ required to cover $\mathcal{X}$ (the so-called Kolmogorov entropy of $X$).
\end{remark}

\begin{remark}  \label{rem_att}
According to the above sources, $\mathcal{M}_0\subset\mathcal{B}^1_0$ is only guaranteed to attract bounded subsets of $\mathcal{B}^1_0$ exponentially fast (in the topology of $\mathcal{H}_0$); i.e.,
\[
\dist_{\mathcal{H}_0}(S_0(t)\mathcal{B}^1_0,\mathcal{M}_0) \leq Q(\|\mathcal{B}^1_0\|_{\mathcal{H}_0})e^{-\omega_2 t},
\]
for some constant $\omega_2>0$.
However, from \cite[Lemma 4.6]{Gal&Shomberg15}, we have that for any bounded subset $B$ of $\mathcal{H}_0$,
\[
\dist_{\mathcal{H}_0}(S_0(t)B,\mathcal{B}^1_0) \leq Q(\|B\|_{\mathcal{H}_0})e^{-\omega_3t},
\]
for some constant $\omega_3>0$ (moreover, this can be see from (\ref{tran-1}) after possibly increasing the radius of $\mathcal{B}^1_0$ in order to contain $\mathcal{U}_0$).
Thus, by the so-called ``transitivity of exponential attraction'' (see Lemma \ref{t:exp-attr}) it follows that (\ref{exp-attn-9}) holds for all non-empty bounded subsets $B$ in $\mathcal{H}_0$.
\end{remark}

Finally, recall

\begin{corollary}
There holds the bound
\begin{equation*}
\dim_{\mathrm{F}}(\mathcal{A}_0,\mathcal{H}_0)\leq \dim_{\mathrm{F}}(\mathcal{M}_0,\mathcal{H}_0).
\end{equation*}
That is, the global attractor $\mathcal{A}_0$ also possesses finite fractal dimension.
\end{corollary}

%%%%%%%%%%%%%%%%%%%%%%%%%%%%%%%%%%%%%%%%%%%%%%
%%%%%%%%%%%%%%%%%%%%%%%%%%%%%%%%%%%%%%%%%%%%%%
\section{Attractors for Problem (A), the $\ep>0$ case}  \label{s:acoustic}
%%%%%%%%%%%%%%%%%%%%%%%%%%%%%%%%%%%%%%%%%%%%%%
%%%%%%%%%%%%%%%%%%%%%%%%%%%%%%%%%%%%%%%%%%%%%%

In this section Problem (A) is discussed. The $\ep=1$ case was already presented in \cite{Frigeri10}. 
The main results presented here follow directly from \cite{Frigeri10} with suitable modifications to account for the perturbation parameter $\ep$ occurring in the equation governing the acoustic boundary condition. 
Generally, we do not need to present the proofs for the case $\ep\in(0,1)$ since the modified proofs follow directly from Frigeri's work \cite{Frigeri10}.
However, in some instances $\ep$ may appear in a crucial way in some parameters.
Hence, it is important to keep track of the required modifications. 
These observations will be explained, where needed, by a remark following the statement of the claim.
Indeed, the radius of the absorbing set $\mathcal{B}_\ep^1$ associated with the semiflow $S_\ep$ here depends on $\ep$ in a crucial way (see Remark \ref{r:time-1}).
Despite this fact, we are still able to show that a family of global attractors exists for each $\ep\in(0,1]$.
Furthermore, under assumptions (\ref{f-assumption-2}), (\ref{g-reg-ass-1})-(\ref{f-reg-ass-3}), we obtain the optimal regularity of the global attractors by showing that they are bounded (though, not uniformly with respect to $\ep$) in a compact subset of the phase space.
Additionally, under assumptions (\ref{f-assumption-2}), (\ref{g-reg-ass-1})-(\ref{f-reg-ass-3}), we also show the existence of a family of exponential attractors.
Consequently, the corresponding global attractor possesses finite fractal dimension.
However, due to a lack of appropriate estimates uniform in the perturbation parameter $\ep$, any robustness/H\"older continuity result for the family of exponential attractors is still out of reach.

The section begins with the functional setting for the problem. 
By using the same arguments in \cite{Frigeri10}, it can easily be shown that, for each $\ep\in(0,1]$, Problem (A) possesses unique global weak solutions in a suitable phase space, and the solutions depend continuously on the initial data. 
For the reader's convenience, we sketch the main arguments involved in the proofs.   
As with Problem (R), the solutions generate a family of Lipschitz continuous semiflows, now depending on $\ep$, each of which admits a bounded, absorbing, positively invariant set. As in \cite{Frigeri10}, the existence of global attractors follows using Ball's asymptotic compactness methods. 
As mentioned, under suitable assumptions on $f$ and $g$, we will also establish the existence of a family of exponential attractors.

%%%%%%%%%%%%%%%%%%%%%%%%%%%%%%%%%%%%%%%%%%%%%%
\subsection{The functional framework for Problem (A)}
%%%%%%%%%%%%%%%%%%%%%%%%%%%%%%%%%%%%%%%%%%%%%%

The phase space and abstract formulation for the perturbation problem are given in this section. The formulation depends on the parameter $\ep$.

Let 
\[
\mathcal{H}:= H^1(\Omega) \times L^2(\Omega) \times L^2(\Gamma) \times L^2(\Gamma).
\]
The space $\mathcal{H}$ is Hilbert with the norm whose square is given by, for $\zeta=(u,v,\delta,\gamma)\in\mathcal{H}$,
\begin{align}
\|\zeta\|^2_{\mathcal{H}} & := \|u\|^2_1 + \|v\|^2 + \|\delta\|^2_{L^2(\Gamma)} + \|\gamma\|^2_{L^2(\Gamma)}   \notag \\ 
& = \left( \|\nabla u\|^2 + \|u\|^2 \right) + \|v\|^2 + \|\delta\|^2_{L^2(\Gamma)} + \|\gamma\|^2_{L^2(\Gamma)}.   \notag
\end{align}
Let $\ep>0$ and denote by $\mathcal{H}_\ep$ the space $\mathcal{H}$ when endowed with the $\ep$-weighted norm whose square is given by
\begin{align}
\|\zeta\|^2_{\mathcal{H}_\ep} := \|u\|^2_1 + \|v\|^2 + \ep\|\delta\|^2_{L^2(\Gamma)} + \|\gamma\|^2_{L^2(\Gamma)}.   \notag
\end{align}
Let
\[
D(\Delta):= \{ u\in L^2(\Omega) : \Delta u\in L^2(\Omega) \},
\]
and define the set
\[
D(\mathrm{A}_\ep) := \left\{ (u,v,\delta,\gamma)\in D(\Delta) \times H^1(\Omega) \times L^2(\Gamma) \times L^2(\Gamma) : \partial_{\bf{n}}u = \gamma \ \text{on} \ \Gamma \right\}.
\]
Define the linear unbounded operator $\mathrm{A}_\ep:D(\mathrm{A}_\ep)\subset\mathcal{H}_\ep\rightarrow\mathcal{H}_\ep$ by
\[
\mathrm{A}_\ep:=\begin{pmatrix} 0 & 1 & 0 & 0 \\ \Delta-1 & -1 & 0 & 0 \\ 0 & 0 & 0 & 1 \\ 0 & -1 & -\ep & -\ep \end{pmatrix}.
\]
For each $\ep\in(0,1]$, the operator $\mathrm{A}_\ep$ with domain $D(\mathrm{A}_\ep)$ is an infinitesimal generator of a strongly continuous semigroup of contractions on $\mathcal{H}_\ep$, denoted $e^{\mathrm{A}_\ep t}$. According to \cite{Frigeri10}, the $\ep=1$ case follows from \cite[Theorem 2.1]{Beale76}. For each $\ep\in(0,1]$, $\mathrm{A}_\ep$ is dissipative because, for all $\zeta=(u,v,\delta,\gamma)\in D(\mathrm{A}_\ep)$,
\[
\langle \mathrm{A}_\ep\zeta,\zeta \rangle_{\mathcal{H}_\ep} = -\|v\|^2 -\ep\|\gamma\|^2_{L^2(\Gamma)} \leq 0.
\]
Also, the Lax-Milgram theorem can be applied to show that the elliptic system, $(I+\mathrm{A}_\ep)\zeta=\xi$, admits a unique weak solution $\zeta\in D(\mathrm{A}_\ep)$ for any $\xi\in\mathcal{H}_\ep$. Thus, $\mathrm{R}(I+\mathrm{A}_\ep)=\mathcal{H}_\ep$. 

\begin{remark}
Notice that ${\rm{rank}}(\mathrm{R})=3$ while ${\rm{rank}}(\mathrm{A}_\ep)=4$; that is, when $\ep=0$, the operator $\mathrm{R}$ exhibits a ``drop in rank.'' This feature corresponds to the Robin boundary condition given in terms of $u_t$ that results when $\ep=0$ in the boundary condition equation (\ref{acoustic-boundary}).
\end{remark}

For each $\ep\in(0,1]$, the map $\mathcal{G}_\ep:\mathcal{H}_\ep\rightarrow\mathcal{H}_\ep$ given by 
\[
\mathcal{G}_\ep(\zeta):=\begin{pmatrix} 0 \\ -f(u) \\ 0 \\ -\ep g(\delta) \end{pmatrix}
\]
for all $\zeta=(u,v,\delta,\gamma)\in\mathcal{H}_\ep$ is locally Lipschitz continuous because the map $f:H^1(\Omega)\rightarrow L^2(\Omega)$ is locally Lipschitz continuous, and, because of the growth assumption on $g$ given in (\ref{g-assumption-1}), it is easy to see that $g:L^2(\Gamma)\rightarrow L^2(\Gamma)$ is globally Lipschitz continuous. Then Problem (A) may be put into the abstract form in $\mathcal{H}_\ep$
\begin{equation}\label{abstract-acoustic-problem}
\left\{
\begin{array}{l} \displaystyle\frac{\diff\zeta}{\diff t} = \mathrm{A}_\ep\zeta + \mathcal{G}_\ep(\zeta) \\ 
\zeta(0)=\zeta_0 \end{array}
\right.
\end{equation}
where $\zeta=\zeta(t)=(u(t),u_t(t),\delta(t),\delta_t(t))$ and $\zeta_0=(u_0,u_1,\delta_0,\delta_1)\in\mathcal{H}_\ep$, now where $v=u_t$ and $\gamma=\delta_t$ in the sense of distributions.

To obtain the {\em{energy equation}} for Problem (A), multiply (\ref{damped-wave-equation}) by $2u_t$ in $L^2(\Omega)$ and multiply (\ref{acoustic-boundary}) by $2\delta_t$ in $L^2(\Gamma)$, then sum the resulting identities to obtain
\begin{equation}\label{acoustic-energy-3}
\frac{\diff}{\diff t} \left\{ \|\zeta\|^2_{\mathcal{H}_\ep} + 2\int_\Omega F(u) \diff x + 2\ep\int_\Gamma G(\delta) \diff \sigma \right\} + 2\|u_t\|^2 + 2\ep\|\delta_t\|^2_{L^2(\Gamma)} = 0,
\end{equation}
where $F(s)=\int_0^s f(\xi) \diff \xi$ and now $G(s)=\int_0^s g(\xi) \diff \xi$, and $\diff \sigma$ represents the natural surface measure on $\Gamma$. 
Reflecting (\ref{from-f-assumption-2}), there is a constant $\mu_1\in(0,1]$ such that 
\begin{equation}\label{from-g-assumption-2}
2\ep\int_\Gamma G(\delta) \diff \sigma \geq -(1-\mu_1)\ep\|\delta\|^2_{L^2(\Gamma)} - \ep\kappa_g
\end{equation}
for some constant $\kappa_g\geq 0$.
Additionally, from the sign condition (\ref{g-assumption-2}), there holds
\begin{equation}\label{from-g-assumption-1}
\ep\langle g(\delta),\delta \rangle_{L^2(\Gamma)} \ge -(1-\mu_1)\ep\|\delta\|^2_{L^2(\Gamma)} - \ep\kappa_g.
\end{equation}

\begin{lemma}  \label{adjoint-a}
For each $\varepsilon \in (0,1]$, the adjoint of $\mathrm{A}_\ep$, denoted 
$\mathrm{A}_\ep^{\ast }$, is given by 
\begin{equation*}
\mathrm{A}_\ep^{\ast }:= -\begin{pmatrix} 0 & 1 & 0 & 0 \\ \Delta-1 & 1 & 0 & 0 \\ 0 & 0 & 0 & 1 \\ 0 & -1 & -\ep & \ep \end{pmatrix},
\end{equation*}
with domain 
\begin{equation*}
D(\mathrm{A}_\ep^{\ast }):=\{(\chi ,\psi, \phi,\xi )\in D(\Delta) \times H^1(\Omega) \times L^2(\Gamma) \times L^2(\Gamma) : \partial _{\mathbf{n}}\chi = - \xi \ \text{on} \ \Gamma \}.
\end{equation*}
\end{lemma}

\begin{proof}
The proof is a calculation similar to, e.g., \cite[Lemma 3.1]{Ball04}.
\end{proof}

%%%%%%%%%%%%%%%%%%%%%%%%%%%%%%%%%%%%%%%%%%%%%%
\subsection{Well-posedness of Problem (A)}
%%%%%%%%%%%%%%%%%%%%%%%%%%%%%%%%%%%%%%%%%%%%%%

Again, the definition of weak solution is from \cite{Ball77}.

\begin{definition} 
Let $T>0$. A map $\zeta\in C([0,T];\mathcal{H}_\ep)$ is a {\em{weak solution}} of (\ref{abstract-acoustic-problem}) on $[0,T]$ if for each $\xi\in D(A^*_\ep)$ the map $t \mapsto \langle \zeta(t),\xi \rangle_{\mathcal{H}_\ep}$ is absolutely continuous on $[0,T]$ and satisfies, for almost all $t\in[0,T]$,
\begin{equation}  \label{abs-2}
\frac{\diff}{\diff t}\langle \zeta(t),\xi \rangle_{\mathcal{H}_\ep} = \langle \zeta(t),A^*_\ep\xi \rangle_{\mathcal{H}_\ep} + \langle \mathcal{G}_\ep(\zeta(t)),\xi \rangle_{\mathcal{H}_\ep}.
\end{equation}
The map $\zeta$ is a weak solution on $[0,\infty)$ (i.e. a {\em{global weak solution}}) if it is a weak solution on $[0,T]$ for all $T>0$. 
\end{definition}

Following \cite{Ball04}, we provide the equivalent notion of a mild solution.

\begin{definition}
Let $T>0$. A function $\zeta:[0,T]\rightarrow\mathcal{H}_\ep$ is a weak/mild solution of (\ref{abstract-acoustic-problem}) on $[0,T]$ if and only if $\mathcal{G}_\ep(\zeta(\cdot))\in L^1(0,T;\mathcal{H}_\ep)$ and $\zeta$ satisfies the variation of constants formula, for all $t\in[0,T],$
\[
\zeta(t)=e^{\mathrm{A}_\ep t}\zeta_0 + \int_0^t e^{\mathrm{A}_\ep(t-s)}\mathcal{G}_\ep(\zeta(s)) \diff s.
\]
\end{definition}

Again, our notion of weak solution is equivalent to the standard concept of a weak (distributional) solution to Problem (A).
Indeed, since $f:H^{1}\left( \Omega \right) \rightarrow L^{2}\left( \Omega \right) $ is sequentially weakly continuous and continuous and $\left( \zeta_{t},\theta \right) \in C^{1}\left( \left[ 0,T\right] \right) $ for all $\theta \in D\left( A^{\ast }\right) $, and (\ref{abs-2}) is satisfied. 

\begin{definition}  \label{aweak}
A function $\zeta=(u,u_{t},\delta,\delta_t):[0,T]\rightarrow \mathcal{H}_\ep$ is a weak solution of (\ref{abstract-acoustic-problem}) (and, thus of (\ref{damped-wave-equation}), (\ref{Robin-initial-conditions}), (\ref{acoustic-boundary}) and (\ref{acoustic-initial-conditions})) on $[0,T],$ if, for almost all $t\in \left[ 0,T\right],$
\begin{equation*}
\zeta =(u,u_{t},\delta,\delta_t)\in C(\left[ 0,T\right] ;\mathcal{H}_\ep),
\end{equation*}
and, for each $\psi \in H^{1}\left( \Omega \right),$ $\langle u_{t},\psi
\rangle \in C^{1}\left( \left[ 0,T\right] \right) $ with
\begin{equation*}
\frac{\diff}{\diff t}\left\langle u_{t}\left( t\right) ,\psi \right\rangle + \left\langle u_{t}\left( t\right) ,\psi \right\rangle + \left\langle u\left( t\right) ,\psi \right\rangle_1 = -\left\langle f\left( u\left( t\right) \right) ,\psi \right\rangle - \left\langle \delta_t\left( t\right) ,\psi \right\rangle _{L^{2}\left( \Gamma \right) },
\end{equation*}
and, for each $\phi \in L^{2}\left(\Gamma\right),$ $\left( \delta_{t},\phi
\right) \in C^{1}\left( \left[ 0,T\right] \right) $ with
\begin{equation*}
\frac{\diff}{\diff t}\left\langle \delta_{t}\left( t\right) ,\phi \right\rangle_{L^{2}\left( \Gamma \right) } + \left\langle \delta_{t}\left( t\right) ,\phi \right\rangle_{L^{2}\left( \Gamma \right) } + \left\langle \delta\left( t\right) ,\phi \right\rangle _{L^{2}\left( \Gamma \right) } = - \left\langle g\left( \delta\left( t\right) \right) ,\phi\right\rangle_{L^2(\Gamma)}.
\end{equation*}
\end{definition}

Recall from the previous section that $f:H^1(\Omega)\rightarrow L^2(\Omega)$ is sequentially weakly continuous and continuous.
Recall that, by \cite[Proposition 3.4]{Ball04} and Lemma \ref{adjoint-a}, $\langle \zeta,\xi \rangle_{\mathcal{H}_\varepsilon}\in C([0,T])$ for all $\xi\in D(A^*)$.

The definition of strong solution follows.
First, for each $\ep\in(0,1]$, define the space, 
\[
\mathcal{D}_\ep:=\left\{ (u,v,\delta,\gamma)\in H^2(\Omega)\times H^1(\Omega)\times H^{1/2}(\Gamma)\times H^{1/2}(\Gamma):\partial_{\bf{n}}u = \gamma ~~\text{on}~~ \Gamma \right\},
\]
and let $\mathcal{D}_\ep$ be equipped with the $\ep$-weighted norm whose square is given by, for all $\zeta=(u,v,\delta,\gamma)\in\mathcal{D}_\ep$, 
\[
\|\zeta\|^2_{\mathcal{D}_\ep}:=\|u\|^2_2+\|v\|^2_1+\ep\|\delta\|^2_{H^{1/2}(\Gamma)}+\|\gamma\|^2_{H^{1/2}(\Gamma)}.
\]

\begin{definition}  \label{d:acoustic-strong}
Let $\zeta_0 = \left(u_0,u_1,\delta_0,\delta_1\right) \in \mathcal{D}_\ep$: that is, let $\zeta_0\in H^{2}(\Omega )\times H^{1}(\Omega )\times H^{1/2}(\Gamma)\times H^{1/2}(\Gamma)$ be such that the compatibility condition 
\begin{equation*}
\partial _{\bf{n}}u_{0} = \delta_1 \ \text{on} \ \Gamma
\end{equation*}
is satisfied. 
A function $\zeta (t) =(u(t),u_t(t),\delta(t),\delta_t(t))$ is called a (global) strong solution if it is a (global) weak solution in the sense of Definition \ref{aweak} and if it satisfies the following regularity properties:
\begin{equation}  \label{acoustic-regularity-property}
\begin{array}{l}
\zeta \in L^{\infty }([0,\infty);\mathcal{D}_\varepsilon)\quad\text{and}\quad\partial_t\zeta\in L^{\infty }([0,\infty);\mathcal{H}_\varepsilon).
\end{array}
\end{equation}
Therefore, $\zeta(t) = (u(t), u_t(t),\delta(t),\delta_t(t)) $ satisfies the equations (\ref{damped-wave-equation}), (\ref{acoustic-boundary}) and (\ref{acoustic-initial-conditions}) almost everywhere; i.e., is a strong solution.
\end{definition}

Now we give the first main result in this section. 

\begin{theorem}  \label{t:ws-a}
Assume (\ref{f-assumption-1}), (\ref{f-assumption-2}), (\ref{g-assumption-1}),(\ref{g-assumption-2}).
Let $\zeta_0\in\mathcal{H}_\ep$. 
For each $\ep\in(0,1]$, there exists a unique global weak solution $\zeta\in C([0,\infty);\mathcal{H}_\ep)$ to (\ref{abstract-acoustic-problem}). 
For each weak solution, the map 
\begin{equation}\label{acoustic-C1-map}
t \mapsto \|\zeta(t)\|^2_{\mathcal{H}_\ep} + 2\int_\Omega F(u(t)) \diff x + 2\ep\int_\Gamma G(\delta(t)) \diff x 
\end{equation}
is $C^1([0,\infty))$ and the energy equation (\ref{acoustic-energy-3}) holds (in the sense of distributions). 
Moreover, for all $\zeta_0, \xi_0\in\mathcal{H}_\ep$, there exists a positive constant $\nu_1>0$, depending on $\|\zeta_0\|_{\mathcal{H}_\ep}$ and $\|\xi_0\|_{\mathcal{H}_\ep}$, such that for all $t\geq 0$,
\begin{equation}\label{acoustic-continuous-dependence}
\|\zeta(t)-\xi(t)\|_{\mathcal{H}_\ep} \leq e^{\nu_1 t} \|\zeta_0 - \xi_0\|_{\mathcal{H}_\ep}.
\end{equation}
Furthermore, when (\ref{f-assumption-2}) and (\ref{g-reg-ass-1})-(\ref{f-reg-ass-3}) hold, and $\zeta_0\in\mathcal{D}_\ep$, then there exists a unique global strong solution $\zeta\in C([0,\infty);\mathcal{D}_\ep)$ to (\ref{abstract-acoustic-problem}). 
\end{theorem}

\begin{proof}
We only mention the first part of the proof. 
Following \cite[Proof of Theorem 1]{Frigeri10}: The operator $\mathrm{A}_\ep$ is the generator of a $C^0$-semigroup of contractions in $\mathcal{H}_\ep$. 
This follows from \cite{Beale76} and the Lumer-Phillips Theorem. 
Also, by (\ref{f-assumption-1}) and (\ref{g-assumption-1}), the functional $\mathcal{G}_\ep:\mathcal{H}_\ep\rightarrow\mathcal{H}_\ep$ is locally Lipschitz continuous. 
So there is $T^*>0$ and a maximal weak solution $\zeta\in C([0,T^*),\mathcal{H}_\ep)$ (cf. e.g. \cite{Zheng04}).
To show $T^*=+\infty$, observe that integrating the energy identity (\ref{acoustic-energy-3}) over $(0,t)$ yields, for all $t\in[0,T^*),$
\begin{align}
& \|\zeta(t)\|^2_{\mathcal{H}_\ep} + 2\int_\Omega F(u(t)) \diff x + 2\ep\int_\Gamma G(\delta(t)) \diff x + \int_0^t \left( 2\|u_\tau(\tau)\|^2d\tau + 2\ep\|\delta_\tau(\tau)\|^2_{L^2(\Gamma)} \right) d\tau   \notag \\
&= \|\zeta_0\|^2_{\mathcal{H}_\ep} + 2\int_\Omega F(u_0) \diff x + 2\ep\int_\Gamma G(\delta_0) \diff x.    \label{fit-1} 
\end{align}
Applying (\ref{from-f-assumption-2}) and (\ref{from-g-assumption-2}) to (\ref{fit-1}), we find that, for all $t\in[0,T^*)$,
\begin{equation*}
\|\zeta(t)\|_{\mathcal{H}_\ep}\le C(\|\zeta_0\|_{\mathcal{H}_\ep}),
\end{equation*}
with some $C>0$ independent of $t$; which of course means $T^*=+\infty.$
Moreover, we know that when $\zeta_0\in\mathcal{H}_\ep$ is such that $\|\zeta_0\|_{\mathcal{H}_\ep}\le R$ for all $\ep\in(0,1],$ then there holds the uniform bound, for all $t\ge0,$
\begin{equation}  \label{acoustic-bound}
\|\zeta(t)\|_{\mathcal{H}_\ep}\le Q(R).
\end{equation}

The remainder of the proof follows directly from \cite[Theorem 1]{Frigeri10}.
\end{proof}

As above, we formalize the dynamical system associated with Problem (A).

\begin{corollary}  \label{sf-a}
Let $\zeta_0=(u_0,u_1,\delta_0,\delta_1)\in\mathcal{H}_\ep$ and let $u$ and $\delta$ be the unique solution of Problem (A). 
For each $\ep\in(0,1],$ the family of maps $S_\ep=(S_\ep(t))_{t\geq 0}$ defined by 
\begin{align}
&S_\ep(t)\zeta_0(x):=  \notag \\
&(u(t,x,u_0,u_1,\delta_0,\delta_1),u_t(t,x,u_0,u_1,\delta_0,\delta_1),\delta(t,x,u_0,u_1,\delta_0,\delta_1),\delta_t(t,x,u_0,u_1,\delta_0,\delta_1))  \notag
\end{align}
is the semiflow generated by Problem (A). 
The operators $S_\ep(t)$ satisfy
\begin{enumerate}
	\item $S_\ep(t+s)=S_\ep(t)S_\ep(s)$ for all $t,s\geq 0$.
	\item $S_\ep(0)=I_{\mathcal{H}_\ep}$ (the identity on $\mathcal{H}_\ep$)
	\item $S_\ep(t)\zeta_0\rightarrow S_\ep(t_0)\zeta_0$ for every $\zeta_0\in\mathcal{H}_\ep$ when $t\rightarrow t_0$.
\end{enumerate}

Additionally, each mapping $S_\ep(t):\mathcal{H}_\ep\rightarrow\mathcal{H}_\ep$ is Lipschitz continuous, uniformly in $t$ on compact intervals; i.e., for all $\zeta_0, \chi_0\in\mathcal{H}_\ep$, and for each $T\geq 0$, and for all $t\in[0,T]$,
\begin{equation}\label{sf-a-lc}
\|S_\ep(t)\zeta_0-S_\ep(t)\chi_0\|_{\mathcal{H}_\ep} \leq e^{\nu_1 T}\|\zeta_0-\chi_0\|_{\mathcal{H}_\ep}.
\end{equation}
\end{corollary}

\begin{proof}
The proof is not much different from the proof of Corollary \ref{sf-r} above.
The Lipschitz continuity property follows from (\ref{acoustic-continuous-dependence}).
\end{proof}

%%%%%%%%%%%%%%%%%%%%%%%%%%%%%%%%%%%%%%%%%%%%%%
\subsection{Dissipativity of Problem (A)}
%%%%%%%%%%%%%%%%%%%%%%%%%%%%%%%%%%%%%%%%%%%%%%

The dynamical system $(S_\ep(t),\mathcal{H}_\ep)$ is shown to admit a positively invariant, bounded absorbing set in $\mathcal{H}_\ep$. 
The argument follows \cite[Theorem 2]{Frigeri10}.

\begin{lemma}  \label{t:a-abs-set}
Assume (\ref{f-assumption-1}), (\ref{f-assumption-2}), (\ref{g-assumption-1}), and (\ref{g-assumption-2}) hold.
For each $\ep\in(0,1]$, there exists $R_{1}>0$, independent of $\ep$, such that the following holds: for every $R>0$, there exists $t_{1\ep}=t_1(\ep,R) \ge 0$, depending on $\ep$ and $R$, so that, for all $\zeta_0\in\mathcal{H}_\ep$ with $\|\zeta_0\|_{\mathcal{H}_\ep} \le R$ for every $\ep\in(0,1]$, and for all $t\geq t_{1\ep}$,
\begin{equation}  \label{acoustic-radius}
\|S_\ep(t)\zeta_0\|_{\mathcal{H}_\ep} \le R_{1}.
\end{equation}
Furthermore, for each $\ep\in(0,1],$ the set 
\begin{equation}  \label{abss-1}
\mathcal{B}_\ep:=\{ \zeta\in \mathcal{H}_\ep : \|\zeta\|_{\mathcal{H}_\ep} \leq R_{1} \}
\end{equation}
is closed, bounded, absorbing, and positively invariant for the dynamical system $(S_\ep,\mathcal{H}_\ep)$.
\end{lemma}

\begin{remark}  \label{r:time-1}
The proof of Lemma \ref{t:a-abs-set} utilizes Proposition \ref{t:diff-ineq-1}, and as such, the ``time of entry,'' $t_{1\ep},$ may be explicitly calculated in terms of the parameters in (\ref{from-f-assumption-2}) and (\ref{from-g-assumption-2}), and in terms of $R$ of course, as
\begin{align}
t_{1\ep}(\iota) & := \frac{1}{\iota}\left( C_{2}R(1+R^3) + \kappa_f+\ep\kappa_g \right).  \notag
\end{align}
Furthermore, after a calculation similar to the one given in Remark \ref{Robin-time}, we find the radii of $\mathcal{B}_\ep$ to be given by, for all $\ep\in(0,1]$ and $\iota>0,$
\begin{align}
R^2_{1\ep}(\iota) &:= \frac{C_2}{C_1} \left( \kappa_f+\ep\kappa_g+\frac{\iota}{m_1\ep} \right)^{1/2} \left( \kappa_f+\ep\kappa_g+1 + \left( \kappa_f+\ep\kappa_g+\frac{\iota}{m_1\ep} \right)^{3/2} \right).  \notag
\end{align}
Observe, in general, $R_{1\ep}\sim\ep^{-1}.$
However, we now choose $\iota=\iota_\ep=m_1 \ep.$
Then
\[
t_{1\ep}(m_1\ep):=\frac{1}{\ep m_1}\left( 1+\frac{C_2R(1+R^3)}{\kappa_f+\ep\kappa_g} \right).
\]
Hence, $t_{1\ep}\rightarrow+\infty$ as $\ep\rightarrow0$. 
But with such a choice of $\iota$, the radius $R_{1\ep}(m_1\ep)$ has a fixed upper-bound {\em{independent}} of $\ep$,
\begin{equation}  \label{g-attr-bound}
R^2_{1\ep}(m_1\ep):= \frac{C_2}{C_1} \left(\kappa_f+\ep\kappa_g+1\right)^{3/2} \left( 1 + \left( \kappa_f+\ep\kappa_g+1 \right)^{1/2} \right).
\end{equation} 
Compare this, and the limit as $\ep\rightarrow 0$, to the radius of $\mathcal{B}_0$ given in (\ref{reg-radii}).
\end{remark}

%%%%%%%%%%%%%%%%%%%%%%%%%%%%%%%%%%%%%%%%%%%%%%
\subsection{Global attractors for Problem (A)}
%%%%%%%%%%%%%%%%%%%%%%%%%%%%%%%%%%%%%%%%%%%%%%

Concerning the existence of global attractors, it is now assumed that (\ref{g-reg-ass-1}) holds; i.e., $g\equiv0$, (yet it is still assumed that $f$ only satisfy (\ref{f-assumption-1})-(\ref{f-assumption-2})). 
Hence, the corresponding acoustic boundary condition is now 
\begin{equation}\label{acoustic-boundary-2}
\left\{ \begin{array}{ll}
\delta_{tt} + \ep[\delta_t + \delta] = -u_t & \text{on} \ (0,\infty)\times\Gamma \\
\delta_t = \partial_{\bf{n}}u & \text{on} \ (0,\infty)\times\Gamma,
\end{array}\right.
\end{equation}
again supplemented with the initial conditions (\ref{acoustic-initial-conditions}).

By using asymptotic compactness methods, it can be shown, with some modifications to \cite{Frigeri10}, that the semiflow $S_\ep$ admits a global attractor in $\mathcal{H}_\ep$, for each $\ep\in(0,1]$; thus, defining a family of global attractors in $\mathcal{H}_\ep$ (cf. e.g. (\ref{gfam-1})). 
The continuity properties of this family of sets will be developed in the next section. 

The existence of a family of global attractors in $\mathcal{H}_\ep$ admitted by the semiflows $S_\ep$ for Problem (A) follows from \cite[Theorem 3.3]{Ball00} as in \S\ref{s:Robin} for Problem (R). 
We remind the reader that for this result, $g\equiv0,$ and Problem (A) is equipped with the linear boundary condition (\ref{acoustic-boundary-2}).

\begin{theorem}\label{t:acoustic-global}
Assume (\ref{f-assumption-1}), (\ref{f-assumption-2}), and (\ref{g-reg-ass-1}) hold.
For each $\ep\in(0,1]$, the dynamical system $(S_\ep(t),\mathcal{H}_\ep)$ admits a global attractor $\mathcal{A}_\ep$ in $\mathcal{H}_\ep$. 
The global attractor is invariant under the semiflow $S_\ep$ (both positively and negatively) and attracts all nonempty bounded subsets of $\mathcal{H}_\ep$; precisely, 
\begin{enumerate}
\item For each $t\geq 0$, $S_\ep(t)\mathcal{A}_\ep=\mathcal{A}_\ep$, and 
\item For every nonempty bounded subset $B$ of $\mathcal{H}_\ep$,
\[
\lim_{t\rightarrow\infty}{\rm{dist}}_{\mathcal{H}_\ep}(S_\ep(t)B,\mathcal{A}_\ep):=\lim_{t\rightarrow\infty}\sup_{\varphi\in B}\inf_{\theta\in\mathcal{A}_\ep}\|S_\ep(t)\varphi-\theta\|_{\mathcal{H}_\ep}=0.
\]
\end{enumerate}
The global attractor is unique and given by 
\[
\mathcal{A}_\ep=\omega(\mathcal{B}_\ep):=\bigcap_{s\geq 0}{\overline{\bigcup_{t\geq s} S_\ep(t)\mathcal{B}_\ep}}^{\mathcal{H}_\ep}.
\]
Furthermore, $\mathcal{A}_\ep$ is the maximal compact invariant subset in $\mathcal{H}_\ep$.
\end{theorem}

Again, the weak continuity of the semiflow $S_\ep$ in $\mathcal{H}_\ep$ follows from the general result from \cite[Theorem 3.6]{Ball04}.
The asymptotic compactness result is borrowed from \cite{Frigeri10}, and generalized from $\ep=1$ to the case where $\ep\in(0,1]$. 

\begin{lemma}
For each $\ep\in(0,1]$, the semiflow $S_\ep$ is weakly continuous on $\mathcal{H}_\ep$; i.e., for each $t\geq 0$,
\[
S_\ep(t)\zeta_{0n}\rightharpoonup S_\ep(t)\zeta_0  \ \text{in} \  \mathcal{H}_\ep ~\text{when}~\zeta_{0n}\rightharpoonup \zeta_0  \ \text{in} \ \mathcal{H}_\ep.
\]
\end{lemma}

\begin{proof}
For each fixed $\ep\in(0,1],$ the proof from \cite[Theorem 3.6]{Ball04} applies to Problem (A).
\end{proof}

\begin{lemma}  \label{aac}
For each $\ep\in(0,1]$, the semiflow $S_\ep$ is asymptotically compact in $\mathcal{H}_\ep$; i.e., if $\zeta_{0n}=(u_{0n},u_{1n},\delta_{0n},\delta_{1n})$ is any bounded sequence in $\mathcal{H}_\ep$ and if $t_n$ is any sequence such that $t_n\rightarrow\infty$ as $n\rightarrow \infty$, then the sequence $\zeta^{(n)}(t_n) = S_\ep(t_n)\zeta_{0n}$ has a convergent subsequence.
\end{lemma}

\begin{proof}[Proof of Theorem \ref{t:acoustic-global}]
We apply the method of generalized semiflows due to Ball \cite{Ball00,Ball04}.
This result follows directly from the existence of a bounded absorbing set $\mathcal{B}_\ep$, due to Lemma \ref{t:a-abs-set}, and the asymptotic compactness of the semiflow $S_\ep$, proven in Lemma \ref{aac}.
\end{proof}

%%%%%%%%%%%%%%%%%%%%%%%%%%%%%%%%%%%%%%%%%%%%%%
\subsection{Optimal regularity for $\mathcal{A}_\ep$}
%%%%%%%%%%%%%%%%%%%%%%%%%%%%%%%%%%%%%%%%%%%%%%

This section discusses the asymptotic compactness result for the weak solutions to Problem (A), assuming the following (\ref{f-assumption-2}), (\ref{g-reg-ass-1}), (\ref{f-reg-ass-2}), and (\ref{f-reg-ass-3}) hold. 
The results directly follow the presentation in \cite{Frigeri10} with modifications to include the perturbation parameter $\ep\in(0,1].$

\begin{theorem}  \label{optimal-a}
Assume (\ref{f-assumption-2}), (\ref{g-reg-ass-1}), (\ref{f-reg-ass-2}), and (\ref{f-reg-ass-3}) hold.
For each $\ep\in(0,1]$, there exists a closed and bounded subset $\mathcal{U}_\ep\subset \mathcal{D}_\ep$, such that for every nonempty bounded subset $B\subset \mathcal{H}_\ep$,
\begin{equation}
\mathrm{dist}_{\mathcal{H}_\ep}(S_\ep(t)B,\mathcal{U}_\ep) \le Q_\ep(\left\Vert B\right\Vert _{\mathcal{H}_\ep}) e^{-\omega_{4 \ep} t},  \label{tran-2}
\end{equation}
for some nonnegative monotonically increasing function $Q_{\ep}(\cdot)$ and for some positive constant $\omega_{4 \ep}>0$, both depending on $\ep$ where $Q_\ep(\cdot)\sim\ep^{-1}$ and $\omega_{4\ep}\sim\ep^{-1}.$
\end{theorem} 

Hence, we immediately have the following

\begin{corollary}
For each $\ep\in(0,1]$, the global attractor $\mathcal{A}_\ep$ admitted by the semiflow $S_\ep$ satisfies 
\begin{equation*}
\mathcal{A}_\ep\subset\mathcal{U}_\ep.
\end{equation*}
Consequently, for each $\ep\in(0,1]$, the global attractor $\mathcal{A}_\ep$ is bounded in $\mathcal{D}_\ep$ and consists only of strong solutions. 
\end{corollary}

The proof of Theorem \ref{optimal-a} proceeds along the usual lines; whereby decomposing the semiflow $S_\ep$ into two parts, one which decays (exponentially) to zero, and one part which is precompact. 
Recall, since fractional powers of the Laplacian with acoustic boundary conditions are undefined, the precompactness property will be earned through the application of $H^2$-elliptic regularity results. 
Define 
\begin{equation}  \label{beta}
\psi (s):=f(s)+\beta s
\end{equation}
for some constant $\beta \ge \ell_2$ to be determined later (observe, $\psi'(s)\ge 0$ thanks to assumption (\ref{f-reg-ass-3})). 
Set $\Psi (s):=\int_{0}^{s}\psi (\sigma )\diff\sigma$. 
Let $\zeta_0=(u_0,u_1,\delta_0,\delta_1)\in\mathcal{H}_\ep$. 
Let $\zeta(t)=(u(t),u_t(t),\delta(t),\delta_t(t))$ denote the corresponding global solution of Problem (A) on $[0,\infty)$ with the initial data $\zeta_0$. 
We then decompose Problem (A) into the following systems of equations.
For all $t\ge0$, set 
\begin{align}
\zeta(t) & = (u(t),u_t(t),\delta(t),\delta_t(t))  \notag \\ 
& = \left( v(t),v_t(t),\gamma(t),\gamma_t(t) \right) + \left( w(t),w_t(t),\theta(t),\theta_t(t) \right)  \notag \\
& =: \xi(t) + \chi(t)  \notag 
\end{align}
Then $\xi$ and $\chi$ satisfy the IBVPs,
\begin{equation}
\left\{ 
\begin{array}{ll}  \label{pde-v}
v_{tt} + v_{t} - \Delta v + v + \psi (u) - \psi (w) = 0 & \text{in}\quad(0,\infty
)\times \Omega, \\ 
\gamma_{tt} + \ep [ \gamma_t + \gamma ] = -v_t & \text{on}\quad(0,\infty)\times \Gamma, \\ 
\gamma_t = \partial_{\bf{n}}v & \text{on}\quad(0,\infty )\times \Gamma, \\ 
\xi(0) = \zeta_0 & \text{at}\quad\{0\}\times{\overline{\Omega}}, \\ 
\end{array}
\right.\end{equation}
and, respectively, 
\begin{equation}
\left\{ 
\begin{array}{ll}  \label{pde-w}
w_{tt} + w_{t} - \Delta w + w + \psi (w) = \beta u & \text{in}\quad(0,\infty)\times \Omega, \\ 
\theta_{tt} + \ep [ \theta_t + \theta ] = -w_t & \text{on}\quad(0,\infty)\times \Gamma, \\ 
\theta_{t} = \partial_{\bf{n}}w & \text{on}\quad(0,\infty)\times \Gamma, \\ 
\chi(0) = {\bf{0}} & \text{at}\quad\{0\}\times{\overline{\Omega}}.
\end{array} \right.
\end{equation}
In view of Lemmas \ref{t:uniform-bound-w} and \ref{t:uniform-decay} below, we define the one-parameter family of maps, $K_\ep(t):\mathcal{H}_\ep\rightarrow \mathcal{H}_\ep$, by 
\begin{equation*}
K_\ep(t)\zeta_{0} := \xi(t),
\end{equation*}
where $\xi(t)$ is a solution of (\ref{pde-w}). 
With such $\xi(t)$, we may define a second function $\chi(t)$ for all $t\ge 0$ as the solution of (\ref{pde-v}). 
Through the dependence of $\xi$ on $\chi$ and $\zeta_{0}$, the solution of (\ref{pde-v}) defines a one-parameter family of maps, $Z_\ep(t):\mathcal{H}_{\varepsilon}\rightarrow \mathcal{H}_\ep$, defined by 
\begin{equation*}
Z_\ep(t)\zeta_{0} := \chi(t).
\end{equation*}
Notice that if $\xi$ and $\chi$ are solutions to (\ref{pde-v}) and (\ref{pde-w}), respectively, then the function $\zeta := \xi + \chi$ is a solution to the
original Problem (A), for each $\ep\in(0,1]$.

The first lemma shows that the operators $K_\ep$ are bounded in $\mathcal{H}_\ep$, uniformly with respect to $\varepsilon$. 
The result largely follows from the existence of a bounded absorbing set $\mathcal{B}_\ep$ in $\mathcal{H}_\ep$ for $S_\ep$ (recall (\ref{acoustic-bound})).

\begin{lemma}  \label{t:uniform-bound-w} 
For each $\varepsilon \in (0,1]$ and $\zeta _{0}=(u_{0},u_{1},\delta_0,\delta_1) \in \mathcal{H}_\ep$, there exists a unique global weak solution $\chi = (w, w_{t}, \theta, \theta_t) \in C([0,\infty );\mathcal{H}_\ep)$ to problem (\ref{pde-w}) satisfying 
\begin{equation}  \label{bnd-2}
\theta_{t}\in L_{\mathrm{loc}}^{2}([0,\infty )\times \Gamma ).
\end{equation}
Moreover, for all $\zeta_{0}\in \mathcal{H}_\ep$ with $\left\Vert \zeta_{0}\right\Vert _{\mathcal{H}_\ep}\leq R$ for
all $\varepsilon \in (0,1]$, there holds for all $t\geq 0$, 
\begin{equation}  \label{uniform-bound-w}
\Vert K_\ep(t)\zeta _{0}\Vert _{\mathcal{H}_\ep} \le Q(R).
\end{equation}
\end{lemma}

\begin{lemma}  \label{t:Gronwall-bound} 
For each $\ep\in(0,1]$ and for all $\eta >0$, there is a function $Q_{\eta}(\cdot) \sim \eta^{-1}$, such that for every $0\leq s\leq t$, $\zeta_{0} = (u_{0},u_{1},\delta_0,\delta_1) \in \mathcal{B}_{\varepsilon }$, and $\ep\in(0,1]$,
\begin{equation}  \label{Gronwall-bound-0}
\int_{s}^{t} \left( \Vert u_{t}(\tau )\Vert ^{2} + \Vert w_{t}(\tau )\Vert
^{2} + \ep\|\theta(\tau)\|^2_{L^2(\Gamma)} \right) \diff\tau \le \frac{\eta }{2}(t - s) + Q_{\eta}(R),
\end{equation}
where $R>0$ is such that $\left\Vert \zeta _{0}\right\Vert _{\mathcal{H}_{\varepsilon }}\leq R,$ for all $\varepsilon\in(0,1]$.
\end{lemma}

The next result shows that the operators $Z_\ep$ are uniformly decaying to zero in $\mathcal{H}_\ep$, for each $\ep\in(0,1]$.

\begin{lemma}  \label{t:uniform-decay}  
For each $\varepsilon \in (0,1]$ and $\zeta_{0}=(u_{0},u_{1},\delta_0,\delta_1) \in \mathcal{H}_{\varepsilon }$, there exists a unique global weak solution $\xi=(v,v_t,\gamma,\gamma_t)\in C([0,\infty );\mathcal{H}_{\varepsilon })$ to problem (\ref{pde-v}) satisfying 
\begin{equation}  \label{bounded-boundary-3}
\gamma_{t}\in L_{\mathrm{loc}}^{2}([0,\infty )\times \Gamma ).
\end{equation}
Moreover, for all $\zeta_0\in\mathcal{D}_\varepsilon$ with $\left\Vert\zeta_{0}\right\Vert _{\mathcal{H}_{\varepsilon }}\le R$ for all $\varepsilon \in (0,1]$, there is a constant $\omega_{5 \ep} > 0$, depending on $\varepsilon $, as $\omega_{5\ep}\sim\ep$, and there is a positive monotonically increasing function $Q_\ep(\cdot)\sim\ep^{-1}$, such that, for all $t\geq 0$, 
\begin{equation}  \label{uniform-decay}
\Vert Z_{\varepsilon }(t)\zeta_{0}\Vert _{\mathcal{H}_{\varepsilon }} \le
Q_\ep(R)e^{-\omega_{5\ep} t}.
\end{equation}
\end{lemma}

The following lemma establishes the precompactness of the operators $K_{\varepsilon }$, for each $\ep\in(0,1].$

\begin{lemma}  \label{compact-a}
For all $\ep\in(0,1]$, and for each $R>0$ and $\zeta\in\mathcal{D}_\ep$ such that $\|\zeta\|_{\mathcal{D}_\ep}\le R$ for all $\ep\in(0,1]$, there exist constants $\omega_{6\ep},R_{2\ep}>0$, both depending on $\ep$, with $\omega_{6\ep}\sim\ep$ and $R_{2\ep}\sim\ep^{-1}$, in which, for all $t\ge0$, there holds
\begin{equation}  \label{a-exp-attr-1}
\|K_\ep(t)\zeta_0\|^2_{\mathcal{D}_\ep} \le Q(R)e^{-\omega_{6\ep}t} + R_{2\ep}.
\end{equation}
\end{lemma}

%%%%%%%%%%%%%%%%%%%%%%%%%%%%%%%%%%%%%%%%%%%%%%
\subsection{Exponential attractors for Problem (A)}
%%%%%%%%%%%%%%%%%%%%%%%%%%%%%%%%%%%%%%%%%%%%%%

We now turn to the existence of exponential attractors for each $\varepsilon\in(0,1]$ for Problem (A). 
By \cite[Theorem 5]{Frigeri10}, we already know that Problem (A) with $\ep=1$ admits an exponential attractor described by Theorem \ref{eaa} below.
Moreover, the result for the perturbed case follows from \cite{Frigeri10} after suitable modifications to include $\varepsilon\in(0,1]$ appearing in the boundary condition (\ref{acoustic-boundary}). 

\begin{theorem}  \label{eaa} 
Assume (\ref{f-assumption-1}), (\ref{f-assumption-2}), and (\ref{g-reg-ass-1}) hold.
For each $\varepsilon \in (0,1]$, the dynamical system $(S_\ep,\mathcal{H}_\ep)$ associated with Problem (A) admits an exponential attractor $\mathcal{M}_{\varepsilon }$ compact in $\mathcal{H}_{\varepsilon},$ and bounded in $\mathcal{D}_\ep$. 
Moreover, for each $\ep\in(0,1]$ fixed, there hold:

(i) For each $t\geq 0$, $S_{\varepsilon }(t)\mathcal{M}_{\varepsilon}\subseteq \mathcal{M}_{\varepsilon }$.

(ii) The fractal dimension of $\mathcal{M}_{\varepsilon }$ with respect to the metric $\mathcal{H}_{\varepsilon }$ is finite, namely,
\begin{equation*}
\dim_{\mathrm{F}}\left( \mathcal{M}_{\varepsilon },\mathcal{H}_{\varepsilon }\right) \leq C_{\varepsilon }<\infty,
\end{equation*}
for some positive constant $C$ depending on $\varepsilon$.

(iii) There exists a positive constant $\nu_{1\ep}>0$ and a nonnegative monotonically increasing function $Q_{\ep}$ both depending on $\ep$, and where $Q_\ep\sim\ep^{-1}$, such that, for all $t\geq 0$, 
\begin{equation*}
\dist_{\mathcal{H}_{\varepsilon }}(S_{\varepsilon }(t)B,\mathcal{M}_{\varepsilon })\leq Q_{\ep}(\Vert B\Vert _{\mathcal{H}_{\varepsilon }})e^{-\nu_{1\ep} t},
\end{equation*}
for every nonempty bounded subset $B$ of $\mathcal{H}_{\varepsilon }$.
\end{theorem}

As with Problem (R) above, the proof of Theorem \ref{eaa} will follow from the application of an abstract proposition reported specifically for our current case below (see, e.g., \cite[Proposition 1]{EMZ00}, \cite{FGMZ04}, \cite{GGMP05}).

\begin{proposition}  \label{abstract-a}
Assume (\ref{f-assumption-1}), (\ref{f-assumption-2}), and (\ref{g-reg-ass-1}) hold.
Let $(S_{\varepsilon},\mathcal{H}_{\varepsilon}) $ be a dynamical system for each $\varepsilon >0$. 
Assume the following hypotheses hold:

\begin{enumerate}

\item[(H1)] There exists a bounded absorbing set $\mathcal{B}_{\varepsilon
}^{1}\subset \mathcal{D}_{\varepsilon }$ which is positively invariant for $S_{\varepsilon }(t).$ 
More precisely, there exists a time $t_{2\ep}>0,$ which \emph{depends} on $\varepsilon >0$, such that
\begin{equation*}
S_{\varepsilon }(t)\mathcal{B}_{\varepsilon }^{1}\subset \mathcal{B}_{\varepsilon }^{1}
\end{equation*}
for all $t\geq t_{2\ep}$ where $\mathcal{B}_{\varepsilon }^{1}$ is endowed with
the topology of $\mathcal{H}_{\varepsilon }.$

\item[(H2)] There is $t^{\ast }\geq t_{2\ep}$ such that the map $S_{\varepsilon
}(t^{\ast })$ admits the decomposition, for each $\varepsilon \in (0,1]$ and
for all $\zeta_{0},\xi_{0}\in \mathcal{B}_{\varepsilon }^{1}$, 
\begin{equation*}
S_{\varepsilon }(t^{\ast })\zeta_{0}-S_{\varepsilon }(t^{\ast })\xi_{0}=L_{\varepsilon }(\zeta_{0},\xi_{0})+R_{\varepsilon }(\zeta_{0},\xi_{0})
\end{equation*}
where, for some constants $\alpha ^{\ast }\in (0,\frac{1}{2})$ and $\Lambda
^{\ast }=\Lambda ^{\ast }(\Omega ,t^{\ast })\geq 0$ with $\Lambda ^{\ast }$ depending on $\varepsilon >0$, the following hold:
\begin{equation}\label{dd-l-a}
\Vert L_{\varepsilon }(\zeta_{0},\xi_{0})\Vert _{\mathcal{H}_{\varepsilon }}\leq \alpha ^{\ast }\Vert \zeta_{0}-\xi_{0}\Vert _{\mathcal{H}_{\varepsilon }}  
\end{equation}
and
\begin{equation}  \label{dd-k-a}
\Vert R_{\varepsilon }(\zeta_{0},\xi_{0})\Vert _{\mathcal{D}_{\varepsilon }}\leq \Lambda ^{\ast }\Vert \zeta_{0}-\xi_{0}\Vert _{\mathcal{H}_{\varepsilon }}.
\end{equation}

\item[(H3)] The map
\begin{equation*}
(t,U)\mapsto S_{\varepsilon }(t)\zeta:[t^{\ast },2t^{\ast }]\times \mathcal{B}_{\varepsilon }^{1}\rightarrow \mathcal{B}_{\varepsilon }^{1}
\end{equation*}
is Lipschitz continuous on $\mathcal{B}_{\varepsilon }^{1}$ in the topology
of $\mathcal{H}_{\varepsilon}$.
\end{enumerate}

Then, $(S_{\varepsilon },\mathcal{H}_{\varepsilon })$ possesses
an exponential attractor $\mathcal{M}_{\varepsilon }$ in $\mathcal{B}_{\varepsilon }^{1}.$
\end{proposition}

\begin{remark}
As in the case for Problem (R), the basin of exponential attraction is indeed the entire phase space thanks to (\ref{a-exp-attr-1}) below.
This in turn implies the exponential attraction of subsets of $\mathcal{B}^1_\varepsilon$ and the transitivity of exponential attraction (see Proposition \ref{t:exp-attr}).
\end{remark}

\begin{corollary}
For each $\ep\in(0,1]$, the global attractors of Theorem \ref{t:acoustic-global} are bounded in $\mathcal{D}_\ep$.

In addition, there holds
\begin{equation*}
\dim_{\mathrm{F}}(\mathcal{A}_\varepsilon,\mathcal{H}_\varepsilon)\leq \dim_{\mathrm{F}}(\mathcal{M}_\varepsilon,\mathcal{H}_\varepsilon)\le C_\ep
\end{equation*}
for some constant $C>0$, depending on $\ep.$
As a consequence, the global attractors $\mathcal{A}_\varepsilon$ are finite dimensional; however, the dimension of $\mathcal{M}_{\varepsilon }$ is not necessarily uniform with respect to $\varepsilon >0$.
\end{corollary}

To prove Theorem \ref{eaa}, we apply the abstract result expressed in Proposition \ref{abstract-a}.
As a preliminary step, we make an observation on the energy equation (\ref{acoustic-energy-3}) associated with Problem (A).

\begin{lemma}  \label{t:to-H1} 
Conditions (H1), (H2) and (H3) hold for each fixed $\varepsilon \in (0,1]$. 
Moreover, for each $\ep\in(0,1]$, and for each $R>0$ and $\zeta\in\mathcal{D}_\ep$ such that $\|\zeta\|_{\mathcal{D}_\ep}\le R$ for all $\ep\in(0,1]$, there exist constants $\widetilde \omega_{6\ep},\widetilde R_{2\ep}>0$, both depending on $\ep$, with $\widetilde\omega_{6\ep}\sim\ep$ and $\widetilde R_{2\ep}\sim\ep^{-1}$, in which, for all $t\ge0$, there holds
\begin{equation}  \label{a-exp-attr-2}
\|S(t)\zeta_0\|^2_{\mathcal{D}_\ep} \le Q(R)e^{-\widetilde\omega_{6\ep}t/2} + \widetilde R_{2\ep}.
\end{equation}
\end{lemma}

\begin{remark}
The ``time of entry'' of some bounded set $B\subset\mathcal{D}_\ep$ into $\mathcal{B}^1_\ep$ is given by
\[
t_{2\ep}=t_{2}(\|B\|_{\mathcal{H}_\ep},R_\ep) = \max\left\{ \frac{1}{\widetilde\omega_{6\ep}}\ln\left( \frac{Q\left(\|B\|_{\mathcal{H}_\ep}\right)}{R_\ep - \widetilde R_{2\ep}} \right),0 \right\}
\]
where $R_\ep$ is the radius of the absorbing set $\mathcal{B}^1_\ep$ in $\mathcal{D}_\ep$.
Furthermore, both $R_{2 \ep}$ and $t_{2 \ep} \rightarrow+\infty$ as $\ep\rightarrow 0$.
\end{remark}

%%%%%%%%%%%%%%%%%%%%%%%%%%%%%%%%%%%%%%%%%%%%%%
%%%%%%%%%%%%%%%%%%%%%%%%%%%%%%%%%%%%%%%%%%%%%%
\section{The continuity of families of sets}
%%%%%%%%%%%%%%%%%%%%%%%%%%%%%%%%%%%%%%%%%%%%%%
%%%%%%%%%%%%%%%%%%%%%%%%%%%%%%%%%%%%%%%%%%%%%%

This section contains a new abstract theorem (Theorem \ref{t:robustness}) concerning the upper-semicontinuity of a family of sets. 
The key assumption in the theorem involves a comparison of the semiflow corresponding to the unperturbed problem to the semiflow corresponding to the perturbation problem in the topology of the perturbation problem. 
The unperturbed problem is ``fitted'' into the phase space of the perturbed problem through the use of two maps, a {\em{canonical extension}} and {\em{lift}}. This approach for obtaining an upper-semicontinuous family of sets developed in this section is largely motivated by \cite{GGMP05}. In the setting presented in this paper, where the perturbation is isolated to the boundary condition, the upper-semicontinuity result is obtained for a broad range of families of sets and the overall analysis is much simpler. 
The result is then applied to Problem (R) and Problem (A).
In the final part of this section, we also consider the difference between Problem (A) and Problem (R), whereby, this time, we {\em{project}} Problem (A) onto the phase space for Problem (R) for estimate in $\mathcal{H}_0$.

The abstract upper-semicontinuity theorem is developed in this section.

\begin{definition} Given two bounded subsets $A$, $B$ in a Banach space $X$, the {\em Hausdorff (asymmetric) semidistance} between $A$ and $B$, in the topology of $X$, is defined by 
\[
\dist_X(A,B):=\sup_{a\in A}\inf_{b\in B}\|a-b\|_X.
\]
\end{definition} 

Suppose $X_0$ is a Banach space with the norm $\|\chi\|_{X_0}$, for all $\chi\in X_0$, and suppose $Y$ is a Banach space with the norm $\|\psi\|_Y$, for all $\psi\in Y$. For $\ep\in(0,1]$, let $X_\ep$ be the one-parameter family of Banach product spaces 
\[
X_\ep=X_0\times Y
\] 
with the $\ep$-weighted norm whose square is given by 
\[
\|(\chi,\psi)\|^2_{X_\ep}=\|\chi\|^2_{X_0} + \ep\|\psi\|^2_{Y}.
\]
For each $\ep\in(0,1]$, let $S_\ep$ be a semiflow on $X_\ep$ and let $S_0$ be a semiflow on $X_0$. Let $\Pi:X_\ep\rightarrow X_0$ be the {\em{projection}} from $X_\ep$ onto $X_0$; for every subset $B_\ep\subset X_\ep$, 
\[
\Pi B_\ep=B_0\subset X_0
\]
(recall, the correspondence between Problem (A) and Problem (R) indicated by the projection is indicated in Remark \ref{key}).

Define the ``lift'' map to map sets $B_0\subset X_0$ to sets in the product $X_\ep$. With the lift map it is possible to measure the semi-distance between sets from $X_0$ with sets in $X_\ep$, using the topology of $X_\ep$. 

\begin{definition}
Given a map $\mathcal{E}:X_0\rightarrow Y$, locally Lipschitz in $X_0$, the map $\mathcal{L}:X_0\rightarrow X_\ep$ defined by $\chi\mapsto (\chi,\mathcal{E}\chi)$ is called a {\em{lift}} map. 
The map $\mathcal{E}$ is called a {\em{canonical extension}}. 
If $B_0$ is a bounded subset of $X_0$, the set $\mathcal{E}B_0\subset Y$ is called the canonical extension of $B_0$ into $Y$, and the set 
\[
\mathcal{L}B_0=\{(\chi,\psi)\in X_\ep : \chi\in B_0, \ \psi\in \mathcal{E}B_0\}
\]
is called the lift of $B_0$ into $X_\ep$.
\end{definition}

What follows is a general description the type of families that are will be continuous in $X_\ep$.

Let $W_0$ be a bounded subset of $X_0$, and let $S_0$ be a semiflow on $X_0$. 
For each $\ep\in(0,1]$ let $W_\ep$ be a bounded subset of $X_\ep$, and let $S_\ep$ be a semiflow on $X_\ep$. 
Let $T>0$ and define the sets 
\begin{equation}\label{set-family-0}
\mathcal{U}_0 = \bigcup_{t\in[0,T]} S_0(t)W_0
\end{equation}
and
\begin{equation}\label{set-family-1}
\mathcal{U}_\ep = \bigcup_{t\in[0,T]} S_\ep(t)W_\ep.
\end{equation}
Define the family of sets $(\mathbb{U}_\ep)_{\ep\in[0,1]}$ in $X_\ep$ by
\begin{equation}\label{set-family-2}
\mathbb{U}_\ep=\left\{ \begin{array}{ll} \mathcal{U}_\ep & 0<\ep\leq 1 \\ \mathcal{L}\mathcal{U}_0 & \ep=0. \end{array} \right.
\end{equation}

\begin{remark}
The sets $W_0$ and $W_\ep$ are not assumed to be absorbing or positively invariant.
\end{remark}

\begin{theorem}\label{t:robustness}
Suppose that the semiflow $S_0$ is Lipschitz continuous on $X_0$, uniformly in $t$ on compact intervals. 
Suppose the lift map $\mathcal{L}$ satisfies the following: for $T>0$ and for any bounded set $B_\ep$ in $X_\ep$, there exists $M=M(\|B_\ep\|_{X_\ep})>0$, depending on $B_\ep$, $\rho\in(0,1]$, both independent of $\ep$, such that for all $t\in [0,T]$ and $(\chi,\psi)\in B_\ep$,
\begin{equation}\label{robustness}
\|S_\ep(t)(\chi,\psi)-\mathcal{L}S_0(t)\Pi(\chi,\psi)\|_{X_\ep}\leq M\ep^\rho.
\end{equation}
Then the family of sets $(\mathbb{U}_\ep)_{\ep\in[0,1]}$ is upper-semicontinuous in the topology of $X_\ep$; precisely, 
\[
{\rm dist}_{X_\ep}(\mathbb{U}_\ep,\mathbb{U}_0) \le M\ep^\rho.
\]
\end{theorem}

\begin{proof}
To begin,
\[
\dist_{\mathcal{H}_\ep}(\mathbb{U}_\ep,\mathbb{U}_0) =  \sup_{a\in\mathcal{U}_\ep}\inf_{b\in\mathcal{L}\mathcal{U}_0}\|a-b\|_{X_\ep}.
\]
Fix $t\in[0,T]$ and $\alpha\in W_\ep$ so that $a=S_\ep(t)\alpha\in\mathcal{U}_\ep$. Then 
\begin{align}
\inf_{b\in\mathcal{L}\mathcal{U}_0}\|a-b\|_{X_\ep} & = \inf_{\substack{\tau\in [0,T] \\ \theta\in W_0}}\|S_\ep(t)\alpha-\mathcal{L}S_0(\tau)\theta\|_{X_\ep}  \notag \\
& \leq \inf_{\theta\in W_0}\|S_\ep(t)\alpha-\mathcal{L}S_0(t)\theta\|_{X_\ep}.  \notag
\end{align}
Since $S_\ep(t)\alpha=a$, 
\begin{align}
\sup_{\alpha\in W_\ep}\inf_{b\in\mathcal{L}\mathcal{U}_0}\|S_\ep(t)\alpha-b\|_{X_\ep} & \leq \sup_{\alpha\in W_\ep}\inf_{\theta\in W_0}\|S_\ep(t)\alpha-\mathcal{L}S_0(t)\theta\|_{X_\ep}   \notag \\ 
& =\dist_{X_\ep}(S_\ep(t)W_\ep,\mathcal{L}S_0(t)W_0)  \notag \\
& \leq \max_{t\in [0,T]}\dist_{X_\ep}(S_\ep(t)W_\ep,\mathcal{L}S_0(t)W_0). \notag
\end{align}
Thus, 
\begin{align}
\sup_{t\in [0,T]}\sup_{\alpha\in W_\ep}\inf_{b\in\mathcal{L}\mathcal{U}_0} \|S_\ep(t)\alpha-b\|_{X_\ep} \leq \max_{t\in [0,T]}\dist_{X_\ep}(S_\ep(t)W_\ep,\mathcal{L}S_0(t)W_0),  \notag
\end{align}
and 
\begin{align}
\sup_{a\in\mathcal{U}_\ep}\inf_{b\in\mathcal{L}\mathcal{U}_0}\|a-b\|_{X_\ep} & \leq \sup_{t\in [0,T]}\sup_{\alpha\in W_\ep}\inf_{b\in\mathcal{L}\mathcal{U}_0} \|S_\ep(t)\alpha-b\|_{X_\ep}  \notag \\
& \leq \max_{t\in [0,T]}\dist_{X_\ep}(S_\ep(t) W_\ep,\mathcal{L}S_0(t)W_0)  \notag \\
& \leq \max_{t\in [0,T]} \sup_{\alpha\in W_\ep}\inf_{\theta\in W_0}\|S_\ep(t)\alpha-\mathcal{L}S_0(t)\theta\|_{X_\ep}.  \notag
\end{align}

The norm is then expanded
\begin{align}
\|S_\ep(t)\alpha-\mathcal{L}S_0(t)\theta\|_{X_\ep} \leq \|S_\ep(t)\alpha & - \mathcal{L}S_0(t)\Pi\alpha\|_{X_\ep}  \notag \\
& + \|\mathcal{L}S_0(t)\Pi\alpha-\mathcal{L}S_0(t)\theta\|_{X_\ep} \label{4-1}
\end{align}
so that by the assumption described in (\ref{robustness}), there is a constant $M>0$ such that for all $t\in[0,T]$ and for all $\alpha\in W_\ep$,
\[
\|S_\ep(t)\alpha-\mathcal{L}S_0(t)\Pi\alpha\|_{X_\ep} \leq M\ep^\rho.
\]
Expand the square of the norm on the right hand side of (\ref{4-1}) to obtain, for $\Pi\alpha=\Pi(\alpha_1,\alpha_2)=\alpha_1\in X_0$ and $\theta\in X_0$,
\begin{equation}\label{triangle-r}
\|\mathcal{L}S_0(t)\Pi\alpha-\mathcal{L}S_0(t)\theta\|^2_{X_\ep} = \|S_0(t)\Pi\alpha-S_0(t)\theta\|^2_{X_0} + \ep\|\mathcal{E}S_0(t)\Pi\alpha - \mathcal{E}S_0(t)\theta\|^2_Y.
\end{equation}
By the local Lipschitz continuity of $\mathcal{E}$ on $X_0$, and by the local Lipschitz continuity of $S_0$ on $X_0$, there is $L>0$, depending on $W_0$, but independent of $\ep$, such that (\ref{triangle-r}) can be estimated by 
\[
\|\mathcal{L}S_0(t)\Pi\alpha-\mathcal{L}S_0(t)\theta\|^2_{X_\ep} \leq L^2(1+\ep)\|\Pi\alpha-\theta\|^2_{X_0}.
\]
Hence, (\ref{4-1}) becomes
\[
\|S_\ep(t)\alpha-\mathcal{L}S_0(t)\theta\|_{X_\ep} \leq M\ep^\rho + L\sqrt{1+\ep}\|\Pi\alpha-\theta\|_{X_0}
\]
and 
\[
\inf_{\theta\in W_0}\|S_\ep(t)\alpha-\mathcal{L}S_0(t)\theta\|_{X_\ep} \leq M\ep^\rho + L\sqrt{1+\ep}\inf_{\theta=\Pi\alpha}\|\Pi\alpha-\theta\|_{X_0}.
\]
Since $\Pi\alpha\in\Pi W_\ep=W_0$, then it is possible to choose $\theta\in W_0$ to be $\theta=\Pi\alpha$. Therefore, 
\[
\dist_{X_\ep}(\mathbb{U}_\ep,\mathbb{U}_0) = \sup_{\alpha\in W_\ep}\inf_{\theta\in W_0}\|S_\ep(t)\alpha-\mathcal{L}S_0(t)\theta\|_{X_\ep} \leq M\ep^\rho.
\]
This establishes the upper-semicontinuity of the sets $\mathbb{U}_\ep$ in $X_\ep$. 
\end{proof}

\begin{remark}
The upper-semicontinuous result given in Theorem \ref{t:robustness} is reminiscent of robustness results (cf. \cite{GGMP05}) in-so-far as we obtain explicit control over the semidistance in terms of the perturbation parameter $\ep$. 
\end{remark}

%%%%%%%%%%%%%%%%%%%%%%%%%%%%%%%%%%%%%%%%%%%%%%
\subsection{The upper-semicontinuity of the family of global attractors for the model problems}
%%%%%%%%%%%%%%%%%%%%%%%%%%%%%%%%%%%%%%%%%%%%%%

The goal of this section is to show that the assumptions of Theorem \ref{t:robustness} are meet. The conclusion is that the family of global attractors for the model problem are upper-semicontinuous.

First, compared to the previous section, $X_0=\mathcal{H}_0$, $Y=L^2(\Gamma)\times L^2(\Gamma)$, and $X_\ep=X_0\times Y=\mathcal{H}_\ep$. Recall that by equation (\ref{S0-Lipschitz-continuous}), $S_0$ is locally Lipschitz continuous on $\mathcal{H}_0$. Define the projection $\Pi:\mathcal{H}_\ep\rightarrow\mathcal{H}_0$ by 
\[
\Pi(u,v,\delta,\gamma) = (u,v);
\]
thus, for every subset $E_\ep\subset\mathcal{H}_\ep$, $\Pi E_\ep=E_0\subset\mathcal{H}_0$. Define the canonical extension $\mathcal{E}:\mathcal{H}_0\rightarrow L^2(\Gamma)\times L^2(\Gamma)$ by, for all $(u,v)\in\mathcal{H}_0$,
\[
\mathcal{E}(u,v)=(0,-u).
\]
Clearly, $\mathcal{E}$ is locally Lipschitz on $\mathcal{H}_0$. Then the lift map, $\mathcal{L}:\mathcal{H}_0\rightarrow\mathcal{H}_\ep$, is defined, for any bounded set $E_0$ in $\mathcal{H}_0$, by
\[
\mathcal{L}E_0:=\{(u,v,\delta,\gamma)\in\mathcal{H}_\ep:(u,v)\in E_0, \delta=0, \gamma=-u \}.
\]
Recall that $\gamma=\delta_t$ in distributions, so the Robin boundary condition---in $u_t$---is achieved by differentiating the last equation with respect to $t$; hence, $\delta_{tt}=-u_t$. The condition $\delta=0$ highlights the drop in rank the Robin problem possesses when $\ep=0$ compared to the acoustic boundary condition when $\ep>0$.

The main result of the section is 

\begin{theorem}  \label{upper}
Assume (\ref{f-assumption-1}), (\ref{f-assumption-2}), and (\ref{g-reg-ass-1}) hold.
Let $\mathcal{A}_0$ denote the global attractor corresponding to Problem (R) and for each $\ep\in(0,1]$, let $\mathcal{A}_\ep$ denote the global attractor corresponding to Problem (A). 
The family of global attractors $(\mathbb{A}_\ep)_{\ep\in[0,1]}$ in $\mathcal{H}_\ep$ defined by
\[
\mathbb{A}_\ep=\left\{ \begin{array}{ll} \mathcal{A}_\ep & 0<\ep\leq 1 \\ \mathcal{L}\mathcal{A}_0 & \ep=0. \end{array} \right.
\]
is upper-semicontinuous in $\mathcal{H}_\ep$, with explicit control over semi-distances in terms of $\ep$. 
(Note: we are not claiming that $\mathcal{LA}_0$ is a global attractor for Problem (R) in $\mathcal{H}_\ep.$)
\end{theorem}

The following claim establishes the assumption made in equation (\ref{robustness}).
The claim indicates that trajectories on $\mathcal{A}_0$ and $\mathcal{A}_\ep$, with the same initial data, may be estimated, on compact time intervals and in the topology of $\mathcal{H}_\ep$, by a constant depending on the radii of the absorbing sets $\mathcal{B}_\ep$ and by the perturbation parameter $\ep$ to some power (recall Remark \ref{r:time-1}, these radii are bounded independent of $\ep$).

\begin{lemma}  \label{compare}
Let $T>0$. 
There is a constant $\Lambda_1>0$, independent of $\ep$, such that, for all $t\in[0,T]$ and for all $\zeta_0\in \mathcal{A}_\ep$,
\begin{equation}  \label{robust-7}
\|S_\ep(t)\zeta_0 - \mathcal{L}S_0(t)\Pi\zeta_0\|_{\mathcal{H}_\ep} \le \Lambda_1\sqrt{\ep}.
\end{equation}
\end{lemma}

\begin{proof}
Let $u$ denote the weak solution of Problem (A) corresponding to the initial data $\zeta_0=(u_0,u_1,\delta_0,\delta_1)\in \mathcal{A}_\ep$, and let $\bar u$ denote the weak solution of Problem (R) corresponding to the initial data $\Pi\zeta_0=(u_0,u_1)\in \mathcal{A}_0$. 
Rewrite the Robin boundary condition as the following system in $\bar u$ and $\bar\delta$,
\begin{equation}\label{Robin-system-1}
\left\{\begin{array}{l} \bar\delta_t = -\bar u \\ 
\bar\delta_t = \partial_{\bf{n}}{\bar{u}}.
\end{array}\right.
\end{equation}
To compare the Robin problem with the acoustic problem in $\mathcal{H}_\ep$, the first equation is differentiated with respect to $t$ and the corresponding system is equipped with initial conditions,
\[
\left\{\begin{array}{l} \bar\delta_{tt} = -\bar u_t \\ 
\bar\delta_t = \partial_{\bf{n}}{\bar{u}} \\
\bar\delta(0,\cdot)=0, \ \bar\delta_t(0,\cdot)=-u_0. 
\end{array}\right.
\]
Observe that through the definition of the lift map, \[
(\bar\delta(0,\cdot),\bar\delta_t(0,\cdot))=\mathcal{E}(u(0,\cdot),u_t(0,\cdot))=(0,-u_0).
\]
 
Let $z=u-\bar u$ and $w=\delta-\bar\delta$; hence, $z$ and $w$ satisfy the system
\begin{equation}\label{z-difference}
\left\{\begin{array}{ll} z_{tt} + z_t - \Delta z + z + f(u) - f(\bar u) = 0 & \text{in} \ (0,\infty)\times\Omega \\
z(0,\cdot)=0, \ z_t(0,\cdot)=0 & \text{on} \ \{0\}\times\Omega \\ 
w_{tt} + \ep(w_t + w) - \ep(\bar\delta_t + \bar\delta) = -z_t & \text{on} \ (0,\infty)\times\Gamma \\
w_t = \partial_{\bf{n}}z & \text{on} \ (0,\infty)\times\Gamma \\
\ep w(0,\cdot)=\ep\delta_0, \ w_t(0,\cdot)=\ep(\delta_1+u_0) & \text{on} \ \{0\}\times\Gamma. \end{array}\right.
\end{equation}
Multiply equation (\ref{z-difference})$_1$ by $2z_t$ in $L^2(\Omega)$ to obtain
\begin{align}
\frac{\diff}{\diff t} & \|z_t\|^2 + 2\|z_t\|^2 + \frac{\diff}{\diff t}\|\nabla z\|^2 - 2\left\langle \partial_{\bf{n}}z,z_t \right\rangle_{L^2(\Gamma)}  \notag \\
& + \frac{\diff}{\diff t}\|z\|^2 + 2\langle f(u)-f(\bar u),z_t \rangle = 0. \label{robust-1}
\end{align}
Multiply equation (\ref{z-difference})$_3$ by $2w_t$ in $L^2(\Gamma)$, to obtain
\begin{align}
\frac{\diff}{\diff t} & \|w_t\|^2_{L^2(\Gamma)} + 2\ep\|w_t\|^2_{L^2(\Gamma)} + \ep\frac{\diff}{\diff t}\|w\|^2_{L^2(\Gamma)}  \notag \\ 
& - 2\ep\langle \bar\delta_t+\bar\delta, w_t \rangle_{L^2(\Gamma)} = - 2\langle z_t, w_t \rangle_{L^2(\Gamma)}.  \label{robust-2}
\end{align}
Since $w_t=\partial_{\bf{n}}z$ on the boundary $\Gamma$, then $ - 2\langle \partial_{\bf{n}}z, z_t \rangle_{L^2(\Gamma)} = -2\langle z_t, w_t \rangle_{L^2(\Gamma)}$. 
Hence, summing equations (\ref{robust-1}) and (\ref{robust-2}) gives, for almost all $t\ge0,$
\begin{align}
\frac{\diff}{\diff t} & \left\{ \|z\|^2_1 + \|z_t\|^2 + \ep\|w\|^2_{L^2(\Gamma)} + \|w_t\|^2_{L^2(\Gamma)} \right\} + 2\|z_t\|^2 + 2\ep\|w_t\|^2_{L^2(\Gamma)}  \notag \\
& - 2\ep\langle \bar\delta_t+\bar\delta, w_t \rangle_{L^2(\Gamma)} + 2\langle f(u)-f(\bar u),z_t \rangle = 0. \label{robust-3}
\end{align}
Estimating the first product yields,
\[
2\ep|\langle \bar\delta_t+\bar\delta,w_t \rangle_{L^2(\Gamma)}| \leq 2\ep^2(\|\bar\delta_t\|^2_{L^2(\Gamma)}+\|\bar\delta\|^2_{L^2(\Gamma)}) + \|w_t\|^2_{L^2(\Gamma)}.
\]
From (\ref{Robin-system-1}), by comparing directly with the solution of the Problem (R), we find that $\bar u\in L^2([0,\infty);H^1(\Omega))$, so by virtue of the trace map, $\bar u_{\mid\Gamma}=\bar\delta_t\in L^2([0,\infty);H^{1/2}(\Gamma)) \hookrightarrow L^2([0,\infty);L^2(\Gamma))$.
By the definition of weak solution for the Robin problem, the map $t\mapsto\|\bar u(t)\|^2_{L^2(\Gamma)}$ is continuous. 
So as $\bar\delta_t(t)=-\bar u(t)$ in $L^2(\Gamma)$, the auxiliary term $\|\bar\delta(t)\|^2_{L^2(\Gamma)}$ is bounded, uniformly in $t$ on compact intervals. 
Since the global attractor $\mathcal{A}_0$ is bounded in $\mathcal{B}_0$, the maps $t\mapsto\|\bar\delta(t)\|_{L^2(\Gamma)}$ and $t\mapsto\|\bar\delta_t(t)\|_{L^2(\Gamma)}$ are bounded, uniformly in $t$ and $\ep\in(0,1]$, by a the radius of $\mathcal{B}_0$, $R_0$. 
Thus, there is a constant $C=C(R_0)>0$, independent of $\ep$, such that, for all $t\in [0,T]$, 
\begin{align}
2\ep|\langle \bar\delta_t+\bar\delta,w_t \rangle_{L^2(\Gamma)}| & \le \ep^2\cdot C(R_0) + \|w_t\|^2_{L^2(\Gamma)}.  \label{robust-4}
\end{align}
Now estimate the remaining product using the local Lipschitz continuity of $f$, 
\begin{equation}\label{robust-5}
2|\langle f(u)-f(\bar u),z_t \rangle| \leq C_\Omega\|z\|^2_1 + \|z_t\|^2,
\end{equation}
where $C_\Omega$ is due to the continuous embedding $H^1(\Omega)\hookrightarrow L^6(\Omega).$
Combining (\ref{robust-3}) (after omitting the two positive terms $2\|z_t\|^2 + 2\ep\|w_t\|^2_{L^2(\Gamma)}$), (\ref{robust-4}), and (\ref{robust-5}) after adding the terms $\ep\|w\|^2_{L^2(\Gamma)} + \|w_t\|^2_{L^2(\Gamma)}$ the the right hand side, produces the differential inequality, which holds for almost all $t\in[0,T]$,
\begin{align}
\frac{\diff}{\diff t} & \left\{ \|z\|^2_1 + \|z_t\|^2 + \ep\|w\|^2_{L^2(\Gamma)} + \|w_t\|^2_{L^2(\Gamma)} \right\}   \notag \\
& \leq \ep^2\cdot C(R_0) + C_\Omega(\|z\|^2_1 + \|z_t\|^2 + \ep\|w\|^2_{L^2(\Gamma)} + \|w_t\|^2_{L^2(\Gamma)}).  \notag
\end{align}
Integrating with respect to $t$ in the compact interval $[0,T]$ yields,
\begin{align}
& \|z(t)\|^2_1 + \|z_t(t)\|^2 + \ep\|w(t)\|^2_{L^2(\Gamma)} + \|w_t(t)\|^2_{L^2(\Gamma)}   \notag \\
& \leq e^{C_\Omega T}(\|z(0)\|^2_1 + \|z_t(0)\|^2 + \ep\|w(0)\|^2_{L^2(\Gamma)} + \|w_t(0)\|^2_{L^2(\Gamma)})  \notag \\ 
& + \ep^2\cdot C(R_0) \left( e^{C_\Omega T}-1 \right).  \label{robust-6}
\end{align}
Because of the initial conditions, $z(0)=z_t(0)=0$,
\[
\ep\|w(0)\|^2_{L^2(\Gamma)} = \ep\|\delta_0\|^2_{L^2(\Gamma)} \quad \text{and}\quad \|w_t(0)\|^2_{L^2(\Gamma)}=\ep\|\delta_1+u_0\|^2_{L^2(\Gamma)}. 
\] 
Since the initial condition $\zeta_0=(u_0,u_1,\delta_0,\delta_1)$ belongs to the bounded attractor $\mathcal{A}_\ep$, $\|w_t(0)\|_{L^2(\Gamma)}\le \ep\cdot R_1$. 
Thus, inequality (\ref{robust-6}) can be written as
\begin{equation}\label{final-1}
\|z(t)\|^2_1 + \|z_t(t)\|^2 + \ep\|w(t)\|^2_{L^2(\Gamma)} + \|w_t(t)\|^2_{L^2(\Gamma)} \le \ep\cdot C(R_0,R_1,\Omega).
\end{equation}

To show (\ref{robust-7}) as claimed, recall that
\begin{align}
\|S_\ep(t)\zeta_0-\mathcal{L}S_0(t)\Pi\zeta_0\|^2_{\mathcal{H}_\ep} &  \notag \\
= \|z(t)\|^2_1 & + \|z_t(t)\|^2 + \ep\|\delta(t)\|^2_{L^2(\Gamma)} + \|\delta_t(t)+\bar u(t)\|^2_{L^2(\Gamma)}.  \label{final-2}
\end{align}
The last two terms are estimated above by 
\begin{equation}  \label{final-3}
\ep\|\delta(t)\|^2_{L^2(\Gamma)} = \ep\|\delta(t)-\bar\delta(t)+\bar\delta(t)\|^2_{L^2(\Gamma)} \leq 2\ep\|w(t)\|^2_{L^2(\Gamma)} + 2\ep\|\bar\delta(t)\|^2_{L^2(\Gamma)},
\end{equation}
and
\begin{align}
\|\delta_t(t)+\bar u(t)\|^2_{L^2(\Gamma)} & = \|\delta_t(t)-\bar\delta_t(t)+\bar\delta_t(t)+\bar u(t)\|^2_{L^2(\Gamma)}   \notag \\
& \leq 2\|w_t(t)\|^2_{L^2(\Gamma)} + 2\|\bar\delta_t(t)+\bar u(t)\|^2_{L^2(\Gamma)}.  \label{final-4}
\end{align}
It follows from (\ref{Robin-system-1}) that, on $\Gamma$, $\bar\delta_t(t)=-\bar u(t)$ so $\|\bar\delta_t(t) + \bar u(t)\|^2_{L^2(\Gamma)}\equiv 0$. Combining inequalities (\ref{final-2}), (\ref{final-3}), and (\ref{final-4}) with (\ref{final-1}), and recalling the bound on $\ep\|\delta(t)\|^2_{L^2(\Gamma)}$, we arrive at
\[
\|S_\ep(t)\zeta_0-\mathcal{L}S_0(t)\Pi\zeta_0\|^2_{\mathcal{H}_\ep} \le \ep\cdot C(R_0,R_1,\Omega).
\]
This establishes equation (\ref{robust-7}).
\end{proof}

The proof of the main result now follows from a direct application of Theorem \ref{t:robustness} to the model problem. 
Theorem \ref{t:robustness} may actually be applied to any family of sets that may be described by (\ref{set-family-0})-(\ref{set-family-2}), which includes the family of global attractors found above. 
However, since the bound on the exponential attractors is {\em{not}} uniform in $\ep$, this upper-semicontinuity result cannot be applied to the corresponding family of exponential attractors.

\begin{proof}[Proof of Theorem \ref{upper}]
Because of the invariance of the global attractors, setting $W_0=\mathcal{A}_0$ in equation (\ref{set-family-0}) and setting $W_\ep=\mathcal{A}_\ep$ in equation (\ref{set-family-1}) produces, respectively, $\mathcal{U}_0=\mathcal{A}_0$ and $\mathcal{U}_\ep=\mathcal{A}_\ep$.
\end{proof}

We conclude this section with some remarks and final observations.

\begin{remark}
One can see from (\ref{final-3}) that the continuity result in Theorem \ref{upper} depends on the topology of $\mathcal{H}_\ep$.
Conversely, the result in \cite{Hale&Raugel88} holds in the corresponding topology with $\ep=1$ fixed.
But recall, the argument made in their work requires more regularity from the solutions.
\end{remark}

Upon reflection with Lemma \ref{compare}, we also obtain the following explicit estimate for the difference of two trajectories, originating from similar initial data, whereby we project Problem (A) onto the phase space for Problem (R).
The result shows that the first two components of the solution to Problem (A) converge to the solution to Problem (R) as $\ep\to 0$, starting with some fixed initial data.

\begin{lemma}
Let $T>0$. 
There is a constant $\Lambda_1>0$, independent of $\ep$, such that, for all $t\in[0,T]$ and for all $\zeta_0\in \mathcal{A}_\ep$,
\begin{equation}  \label{robust-7-70}
\|\Pi S_\ep(t)\zeta_0 - S_0(t)\Pi\zeta_0\|_{\mathcal{H}_0} \le \Lambda_2\sqrt{\ep}.
\end{equation}
\end{lemma}

\begin{proof}
Let us here consider the difference between Problem (R) and Problem (A) whereby this time we project Problem (A) onto the phase space for Problem (R).
The {\em{projected Problem (A)}} is obtained from equations (\ref{damped-wave-equation})-(\ref{Robin-initial-conditions}) and (\ref{acoustic-boundary}), and (\ref{acoustic-initial-conditions2}),
\begin{equation}  \label{proj-A}
\left\{\begin{array}{ll} u_{tt} + u_t - \Delta u + u + f(u) = 0 & \text{in} \ (0,\infty)\times\Omega \\
u(0,\cdot)=u_0, \ u_t(0,\cdot)=u_1 & \text{on} \ \{0\}\times\Omega \\ 
\delta_{tt} = -u_t & \text{on} \ (0,\infty)\times\Gamma \\
\delta_t = \partial_{\bf{n}}u & \text{on} \ (0,\infty)\times\Gamma \\
\delta_t(0,\cdot)=\ep\delta_1-(1-\ep)u_0 & \text{on} \ \{0\}\times\Gamma. \end{array}\right.
\end{equation}
The associated solution operator is denoted $\Pi S_\ep(t).$
(It is important to note that the initial data for the projected Problem (A) is not projected, only the corresponding solution is.)
Observe, the final three equations, (\ref{proj-A})$_3$-(\ref{proj-A})$_5$ may be reduced; they imply that there is an ``arbitrary constant'' $\phi=\phi(x)$ in which there holds on $\Gamma$,
\[
\partial_{\bf{n}}u+u=\phi(x) \quad \text{where} \quad \phi=\ep(\delta_1+u_0).
\]
Let $u$ denote the weak solution of Problem (A) corresponding to the initial data $\zeta_0=(u_0,u_1,\delta_0,\delta_1)\in \mathcal{A}_\ep$, and let $\bar u$ denote the weak solution of Problem (R) corresponding to the initial data $\Pi\zeta_0=(u_0,u_1)\in \mathcal{A}_0$. 
Again, rewrite the Robin boundary condition as in (\ref{Robin-system-1}), letting $z=u-\bar u$ and $w=\delta-\bar\delta$, we find, this time, $z$ and $w$ satisfy the system,
\begin{equation}  \label{z-difference-70}
\left\{\begin{array}{ll} z_{tt} + z_t - \Delta z + z + f(u) - f(\bar u) = 0 & \text{in} \ (0,\infty)\times\Omega \\
z(0,\cdot)=0, \ z_t(0,\cdot)=0 & \text{on} \ \{0\}\times\Omega \\ 
\partial_{\bf{n}}z + z = \ep\phi(x) & \text{on} \ (0,\infty)\times\Gamma. \end{array}\right.
\end{equation}
At this point it is easy to check that, by virtue of the continuous dependence argument for Problem (R) (cf. (\ref{continuous-dependence})), for almost all $t\ge0$, there holds, 
\begin{equation}  \label{proj-p-1}
\frac{\diff}{\diff t} \|\bar\varphi\|^2_{\mathcal{H}_0} \le C \|\bar\varphi\|^2_{\mathcal{H}_0} + \ep\|\phi(x)\|^2.
\end{equation}
Thus, we find that there is a constant $\Lambda_2>0$, such that for all $t\in[0,T]$, (\ref{robust-7-70}) holds.
\end{proof}

\begin{remark}
At this point it is worth contrasting the above results with what we know of the following related perturbation problem where, for $\ep\in(0,1],$
\begin{equation}  \label{abc2}
\left\{ \begin{array}{ll} 
\ep\delta_{tt}+\delta_t+\delta=-u_t \\ 
\delta_t=\partial_{\bf{n}}u & \quad\text{on}\quad (0,\infty)\times\Gamma. \end{array} \right.
\end{equation}
The limit problem is $\partial_{\bf{n}}u+\delta=-u_t$, which could be compared with the dynamic boundary condition 
\begin{equation}  \label{dybc}
\partial_{\bf{n}}\bar u+\bar u_t=0 \quad\text{on}\quad (0,\infty)\times\Gamma,
\end{equation}
as found in \cite[Equation (1.4)]{Gal12}. 
Writing (\ref{dybc}) as the system 
\begin{equation}  \label{dybc2}
\left\{ \begin{array}{ll} 
\bar\delta_t=-\bar u_t \\ 
\bar\delta_t=\partial_{\bf{n}}\bar u & \quad\text{on}\quad (0,\infty)\times\Gamma. \end{array} \right.
\end{equation}
allows us to write the difference when we take the initial conditions
\[
\zeta(0)=\zeta_0=(u_0,u_1,\ep\delta_0,\delta_1)\in\mathcal{H}_\ep,
\]
and
\[
\mathcal{L}\Pi\zeta_0=(u_0,u_1,\ep u_0,-u_1)
\]
into account.
Letting $z=u-\bar u$ and $w=\delta-\bar\delta$ as above, we find
\begin{equation}  \label{z-difference2}
\left\{\begin{array}{ll} z_{tt} + z_t - \Delta z + z + f(u) - f(\bar u) = 0 & \text{in} \quad (0,\infty)\times\Omega \\
z(0,\cdot)=0, \ z_t(0,\cdot)=0 & \text{at} \quad \{0\}\times\Omega \\ 
\ep w_{tt} + w_t + w= -z_t - \bar\delta - \ep\bar\delta_{tt} & \text{on} \quad (0,\infty)\times\Gamma \\
w_t = \partial_{\bf{n}}z & \text{on} \quad (0,\infty)\times\Gamma \\
w(0,\cdot)=\ep(\delta_0-u_0), \ \ep w_t(0,\cdot)=\ep(\delta_1+u_1) & \text{at} \quad \{0\}\times\Gamma. \end{array}\right.
\end{equation}
The first problem that arrises concerns the bound on the term $\bar\delta_{tt}$, uniform in $\ep$ and $t$ on compact intervals.
Since $\bar\delta=-\bar u_{tt}$ by (\ref{dybc2}), such a bound can be obtained from arguments similar to those in \cite[Proof of Lemma 3.16]{Gal&Shomberg15}; however, this in turn necessitates the regularity assumptions such as (\ref{f-assumption-2}), (\ref{f-reg-ass-2})-(\ref{f-reg-ass-3}) for Problem (R). 
Also, another concern comes from the presence of the new term $\bar\delta$ in the right-hand side of (\ref{z-difference2})$_3$.
The function $\bar\delta$ is determined from (\ref{dybc2}) (that is, the transport-type equation (\ref{dybc})) and the initial condition $\bar\delta(0,\cdot)=\ep u_0$.
The $\ep$ present in this initial condition insures we obtain a control like that obtained in (\ref{robust-7}).
This model will be examined in a subsequent article.
\end{remark}

%%%%%%%%%%%%%%%%%%%%%%%%%%%%%%%%%%%%%%%%%%%%%%
%%%%%%%%%%%%%%%%%%%%%%%%%%%%%%%%%%%%%%%%%%%%%%
\section{Conclusions}
%%%%%%%%%%%%%%%%%%%%%%%%%%%%%%%%%%%%%%%%%%%%%%
%%%%%%%%%%%%%%%%%%%%%%%%%%%%%%%%%%%%%%%%%%%%%%

In this article, an upper-semicontinous family of global attractors was constructed for a damped semilinear wave equation possessing a singular perturbation parameter occurring in prescribed acoustic boundary condition. 
The result is obtained with a rather restrictive growth condition on the nonlinear term, and as a result, the global attractors are arrived at after obtaining an asymptotic compactness property of the semiflows.
With the perturbation parameter occurring in the boundary conditions, the semiflow corresponding to the limit problem, Problem (R), is Lipschitz continuous on its phase space. 
This result is utilized at a critical step in the proof of the upper-semicontinuity of a generic family of sets which was applied to the global attractors. 
Another crucial property in the framework of this problem is that the lift map does not require any regularity. 
For example, the global attractor $\mathcal{A}_0$ for Problem (R) is in $H^1(\Omega)\times L^2(\Omega)$ whereby its lift is in $H^1(\Omega)\times L^2(\Omega)\times L^2(\Gamma)\times L^2(\Gamma)$; no additional regularity of the attractor is required in order to obtain the upper-semicontinuity result. 
Recall, this is certainly not the case for problems with a perturbation of hyperbolic-relaxation type.

In comparison to the upper-semicontinuity result for the global attractors in \cite{Hale&Raugel88}, the perturbation parameter there occurs as a hyperbolic-relaxation term (see the motivation in \S1). 
The global attractors for the parabolic problem must be in (at least) $H^2(\Omega)\cap H^1_0(\Omega)$ in order for the lift to be well-defined in $H^1(\Omega)\times L^2(\Omega)$. 
Such a regularity result requires stronger assumptions than those initially used here on the nonlinear term. 
In addition, while the parabolic semiflow $S_0$ is Lipschitz continuous in $L^2(\Omega)$, it is not necessarily so in $H^1_0(\Omega)$. 
Because of this, another approach is needed when investigating the continuity of the family of global attractors in $H^1(\Omega)\times L^2(\Omega)$.

The global attractors $\mathcal{A}_\ep$ obtained under the restrictive growth assumptions were shown to be bounded, uniformly with respect to the perturbation parameter $\ep$, in the phase space $\mathcal{H}_\ep.$
With assumptions sufficient to allow the existence of both weak and strong solutions to both Problem (R) and Problem (A), we showed that the family of global attractors, $\mathcal{A}_{\ep}$, $\ep\in[0,1],$ possess optimal regularity; each is bounded in the more regular phase space $\mathcal{D}_\ep.$
However, here, the bound is no longer independent of $\ep.$
With this further regularity on the nonlinear term, we showed that the semiflows $S_\ep$ admit a decomposition into exponentially decaying to zero and uniformly precompact parts.
Since fractional-powers of the Laplacian are undefined for Problem (A), we resorted to other $H^2$-regularity methods.
In a natural way, these results allowed us to show the existence of a bounded absorbing set $\mathcal{B}^1_\ep$ in $\mathcal{D}_\ep;$ a first step to proving the existence of an exponential attractor.
Indeed, further properties of $S_\ep$ are shown: Lipschitz continuity on $[T,2T]\times\mathcal{B}^1_\ep$, for some $0<T<\infty,$ and a squeezing property.
The existence of a family of exponential attractors (also upper-semicontinuous) means the corresponding family of global attractors (for each $\ep\in[0,1]$) possesses finite fractal dimension.
Through various estimates which depend on $\ep$ in a crucial way, the radius of the absorbing set $\mathcal{B}^1_\ep$ depends on $\ep$; moreover, the fractal dimension of $\mathcal{A}_\ep$ and $\mathcal{M}_\ep$ is not necessarily uniform in $\ep.$

Two important practical results stemming from this work are the following: First, the nature of the upper-semicontinuity result of the attractors, as presented here, means Problem (A) is a ``relaxation'' of Problem (R).
Precisely, in light of Lemma \ref{compare}, Problem (A) can be interpreted as an approximation of Problem (R), and, in this case, the difference between corresponding trajectories, on compact time intervals, is controlled explicitly in terms of $\sqrt{\ep}$.
Secondly, the finite dimensionality of the attractors means the infinite-dimensional dynamics inherent in the systems associated with Problem (R) and Problem (A) can be reduced to a finite-dimensional system of ODEs.

%%%%%%%%%%%%%%%%%%%%%%%%%%%%%%%%%%%%%%%%%%%%%%
%%%%%%%%%%%%%%%%%%%%%%%%%%%%%%%%%%%%%%%%%%%%%%
\appendix
\section{}
%%%%%%%%%%%%%%%%%%%%%%%%%%%%%%%%%%%%%%%%%%%%%%
%%%%%%%%%%%%%%%%%%%%%%%%%%%%%%%%%%%%%%%%%%%%%%

In this section we include some useful results utilized by Problem (R) and Problem (A). 
The first result can be found in \cite[Lemma 2.7] {Belleri&Pata01}.

\begin{proposition}  \label{t:diff-ineq-1}
Let $X$ be an arbitrary Banach space, and $Z\subset C([0,\infty);X)$. 
Suppose that there is a functional $E:X\rightarrow\mathbb{R}$ such that, for every $z\in Z$,
\begin{equation}\label{zupper-1}
\sup_{t\geq 0} E(z(t))\geq -r \ \text{and} \ E(z(0))\leq R
\end{equation}
for some $r,R\geq 0$. 
In addition, assume that the map $t\mapsto E(z(t))$ is $C^1([0,\infty))$ for every $z\in Z$ and that for almost all $t\geq 0$, the differential inequality holds
\begin{equation}  \label{id-007}
\frac{\diff}{\diff t} E(z(t)) + m\|z(t)\|^2_X \leq C,
\end{equation}
for some $m>0$, $C\geq 0$, both independent of $z\in Z$. Then, for every $\iota>0$, there exists $t_0\geq 0$, depending on $R$ and $\iota$, such that for every $z\in Z$ and for all $t\geq t_0$,
\begin{equation}\label{wer-2}
E(z(t))\leq \sup_{\xi\in X}\{ E(\xi):m\|\xi\|^2_X\leq C+\iota \}.
\end{equation}
Furthermore, $t_0=(r+R)/\iota$.
\end{proposition} 

The following result is the so-called transitivity property of exponential attraction from \cite[Theorem 5.1]{FGMZ04}.

\begin{proposition}  \label{t:exp-attr}
Let $(\mathcal{X},d)$ be a metric space and let $S_t$ be a semigroup acting on this space such that 
\[
d(S_t m_1,S_t m_2) \leq C e^{Kt} d(m_1,m_2),
\]
for appropriate constants $C$ and $K$. Assume that there exists three subsets $M_1$,$M_2$,$M_3\subset\mathcal{X}$ such that 
\[
{\rm{dist}}_\mathcal{X}(S_t M_1,M_2) \leq C_1 e^{-\alpha_1 t} \quad\text{and}\quad{\rm{dist}}_\mathcal{X}(S_t M_2,M_3) \leq C_2 e^{-\alpha_2 t}.
\]
Then 
\[
{\rm{dist}}_\mathcal{X}(S_t M_1,M_3) \leq C' e^{-\alpha' t},
\]
where $C'=CC_1+C_2$ and $\alpha'=\frac{\alpha_1\alpha_1}{K+\alpha_1+\alpha_2}$.
\end{proposition}

%%%%%%%%%%%%%%%%%%%%%%%%%%%%%%%%%%%%%%%%%
%%%%%%%%%%%%%%%%%%%%%%%%%%%%%%%%%%%%%%%%%
\section*{Acknowledgments}
%%%%%%%%%%%%%%%%%%%%%%%%%%%%%%%%%%%%%%%%%
%%%%%%%%%%%%%%%%%%%%%%%%%%%%%%%%%%%%%%%%%

The author gratefully acknowledges Sergio Frigeri for his consulting on this project.

%%%%%%%%%%%%%%%%%%%%%%%%%%%%%%%%%%%%%%%%%%%%%%
%%%%%%%%%%%%%%%%%%%%%%%%%%%%%%%%%%%%%%%%%%%%%%
%%%%%%%%%%%%%%%%%%%%%%%%%%%%%%%%%%%%%%%%%%%%%%
%%%%%%%%%%%%%%%%%%%%%%%%%%%%%%%%%%%%%%%%%%%%%%
\bigskip

%\bibliographystyle{amsplain}
%\bibliography{/Users/josephshomberg/Dropbox/LaTeX/REFERENCES}

\begin{thebibliography}{10}

\bibitem{Babin&Vishik92}
A.~V. Babin and M.~I. Vishik, \emph{Attractors of evolution equations},
  North-Holland, Amsterdam, 1992.

\bibitem{Ball77}
J.~M. Ball, \emph{Strongly continuous semigroups, weak solutions, and the
  variation of constants formula}, Proc. Amer. Math. Soc. \textbf{63} (1977),
  no.~2, 370--373.

\bibitem{Ball00}
\bysame, \emph{Continuity properties and global attractors of generalized
  semiflows and the {N}avier-{S}tokes equations}, Nonlinear Science \textbf{7}
  (1997), no.~5, 475--502, Corrected version appears in the book {\em
  Mechanics: From Theory to Computation}, Springer-Verlag, New York 447--474,
  2000.

\bibitem{Ball04}
\bysame, \emph{Global attractors for damped semilinear wave equations},
  Discrete Contin. Dyn. Syst. \textbf{10} (2004), no.~2, 31--52.

\bibitem{Beale76}
J.~T. Beale, \emph{Spectral properties of an acoustic boundary condition},
  Indiana Univ. Math. J. \textbf{25} (1976), 895--917.

\bibitem{Beale&Rosencrans74}
J.~T. Beale and S.~I. Rosencrans, \emph{Acoustic boundary conditions}, Bull.
  Amer. Math. Soc. \textbf{80} (1974), 1276--1278.

\bibitem{Morante79}
Aldo Belleni-Morante, \emph{Applied semigroups and evolution equations},
  Clarendon Press, Oxford, 1979.

\bibitem{Belleri&Pata01}
Veronica Belleri and Vittorino Pata, \emph{Attractors for semilinear strongly
  damped wave equations on $\mathbb{R}^3$}, Discrete Contin. Dynam. Systems
  \textbf{7} (2001), no.~4, 719--735.

\bibitem{BGM10}
A.~Bonfoh, M.~Grasselli, and A.~Miranville, \emph{Inertial manifolds for a
  singular perturbation of the {C}ahn-{H}illiard-{G}urtin equation}, Topol.
  Methods Nonlinear Anal. \textbf{35} (2010), 155--185.

\bibitem{Bonfoh11}
Ahmed Bonfoh, \emph{The singular limit dynamics of the phase-field equations},
  Ann. Mat. Pura Appl. (4) \textbf{190} (2011), no.~1, 105--144.

\bibitem{CGG11}
Cecilia Cavaterra, Ciprian Gal, and Maurizio Grasselli, \emph{{C}ahn-{H}illiard
  equations with memory and dynamic boundary conditions}, Asymptot. Anal.
  \textbf{71} (2011), no.~3, 123--162.

\bibitem{CEL02}
Igor Chueshov, Matthias Eller, and Irena Lasiecka, \emph{On the attractor for a
  semilinear wave equation with critical exponent and nonlinear boundary
  dissipation}, Comm. Partial Differential Equations \textbf{27} (2002),
  no.~9-10, 1901--1951.

\bibitem{CEL04-2}
\bysame, \emph{Attractors and their structure for semilinear wave equations
  with nonlinear boundary dissipation}, Bol. Soc. Parana. Mat. \textbf{22}
  (2004), no.~1, 38--57.

\bibitem{CEL04}
\bysame, \emph{Finite dimensionality of the attractor for a semilinear wave
  equation with nonlinear boundary dissipation}, Comm. Partial Differential
  Equations \textbf{29} (2004), no.~11-12, 1847--1876.

\bibitem{CFL01}
A.~T. Cousin, C.~L. Frota, and N.~A. Larkin, \emph{Global solvability and
  asymptotic behaviour of a hyperbolic problem with acoustic boundary
  conditions}, Funkcialaj Ekvacioj \textbf{44} (2001), 471--485.

\bibitem{EMZ00}
Messoud Efendiev, Alain Miranville, and Sergey Zelik, \emph{Exponential
  attractors for a nonlinear reaction-diffusion systems in $\mathbb{R}^3$}, C.
  R. Acad. Sci. Paris S\`{e}r. I Math. \textbf{330} (2000), no.~8, 713--718.

\bibitem{FGMZ04}
P.~Fabrie, C.~Galusinski, A.~Miranville, and S.~Zelik, \emph{Uniform
  exponential attractors for singularly perturbed damped wave equations},
  Discrete Contin. Dyn. Syst. \textbf{10} (2004), no.~2, 211--238.

\bibitem{Frigeri10}
Sergio Frigeri, \emph{Attractors for semilinear damped wave equations with an
  acoustic boundary condition}, J. Evol. Equ. \textbf{10} (2010), no.~1,
  29--58.

\bibitem{Gal08}
Ciprian~G. Gal, \emph{Robust exponential attractors for a conserved
  {C}ahn-{H}illiard model with singularly perturbed boundary conditions},
  Discrete Contin. Dyn. Syst. \textbf{7} (2008), no.~4, 819--836.

\bibitem{Gal12-2}
\bysame, \emph{On a class of degenerate parabolic equations with dynamic
  boundary conditions}, J. Differential Equations \textbf{253} (2012),
  126--166.

\bibitem{Gal12}
\bysame, \emph{Sharp estimates for the global attractor of scalar
  reaction-diffusion equations with a {W}entzell boundary condition}, J.
  Nonlinear Sci. \textbf{22} (2012), no.~1, 85--106.

\bibitem{GGG03}
Ciprian~G. Gal, Gis\`{e}le~Ruiz Goldstein, and Jerome~A. Goldstein,
  \emph{Oscillatory boundary conditions for acoustic wave equations}, J. Evol.
  Equ. \textbf{3} (2003), 623--635.

\bibitem{Gal&Grasselli12}
Ciprian~G. Gal and M.~Grasselli, \emph{Singular limit of viscous
  {C}ahn-{H}illiard equations with memory and dynamic boundary conditions},
  DCDS-S, to appear.

\bibitem{GGM08-2}
Ciprian~G. Gal, Maurizio Grasselli, and Alain Miranville, \emph{Robust
  exponential attractors for singularly perturbed phase-field equations with
  dynamic boundary conditions}, NoDEA Nonlinear Differential Equations Appl.
  \textbf{15} (2008), no.~4-5, 535--556.

\bibitem{Gal&Shomberg15}
Ciprian~G. Gal and Joseph~L. Shomberg, \emph{Hyperbolic relaxation of reaction
  diffusion equations with dynamic boundary conditions}, Quart. Appl. Math.
  (2015), to appear.

\bibitem{GGMP05-CH3D}
S.~Gatti, M.~Grasselli, A.~Miranville, and V.~Pata, \emph{Hyperbolic relaxation
  of the viscous {C}ahn-{H}illiard equation in 3{D}}, Math. Models Methods
  Appl. Sci. \textbf{15} (2005), no.~2, 165--198.

\bibitem{GGMP05-CH1D}
\bysame, \emph{On the hyperbolic relaxation of the one-dimensional
  {C}ahn-{H}illiard equation}, J. Math. Anal. Appl. \textbf{312} (2005),
  230--247.

\bibitem{GGMP05}
\bysame, \emph{A construction of a robust family of exponential attractors},
  Proc. Amer. Math. Soc. \textbf{134} (2006), no.~1, 117--127.

\bibitem{GMPZ10}
S.~Gatti, A.~Miranville, V.~Pata, and S.~Zelik, \emph{Continuous families of
  exponential attractors for singularly perturbed equations with memory}, Proc.
  Roy. Soc. Edinburgh Sect. A \textbf{140} (2010), 329--366.

\bibitem{Goldstein85}
Jerome~A. Goldstein, \emph{Semigroups of linear operators and applications},
  Oxford Mathematical Monographs, Oxford University Press, New York, 1985.

\bibitem{Hale&Raugel88}
J.~Hale and G.~Raugel, \emph{Upper semicontinuity of the attractor for a
  singularly perturbed hyperbolic equation}, J. Differential Equations
  \textbf{73} (1988), no.~2, 197--214.

\bibitem{Hale88}
Jack~K. Hale, \emph{Asymptotic behavior of dissipative systems}, Mathematical
  Surveys and Monographs - No. 25, American Mathematical Society, Providence,
  1988.

\bibitem{Hebey99}
Emmanuel Hebey, \emph{Nonlinear analysis on manifolds: Sobolev spaces and
  inequalities}, Courant Institute of Mathematical Sciences, American
  Mathematical Society, Providence, 1999.

\bibitem{Lions69}
J.~L. Lions, \emph{Quelques m\'ethodes de r\'esolution des probl\`emes aux
  limites non lin\'eaires}, Dunod, Paris, 1969.

\bibitem{Lions&Magenes72}
J.~L. Lions and E.~Magenes, \emph{Non-homogeneous boundary value problems and
  applications}, vol.~I, Springer-Verlag, Berlin, 1972.

\bibitem{Milani&Koksch05}
Albert~J. Milani and Norbert~J. Koksch, \emph{An introduction to semiflows},
  Monographs and Surveys in Pure and Applied Mathematics - Volume 134, Chapman
  \& Hall/CRC, Boca Raton, 2005.

\bibitem{Miranville&Zelik02}
Alain Miranville and Sergey Zelik, \emph{Robust exponential attractors for
  singularly perturbed phase-field type equations}, Electron. J. Differential
  Equations \textbf{2002} (2002), no.~63, 1--28.

\bibitem{Mora&Morales89-2}
Xavier Mora and Joan Sol\`{a}-Morales, \emph{Inertial manifolds of damped
  semilinear wave equations}, Mod\'{e}lisation math\'{e}matique et analyse
  num\'{e}rique \textbf{23} (1989), no.~3, 489--505.

\bibitem{Mugnolo10}
Delio Mugnolo, \emph{Abstract wave equations with acoustic boundary
  conditions}, Math. Nachr. \textbf{279} (2006), no.~3, 299--318.

\bibitem{Pata&Zelik06}
Vittorino Pata and Sergey Zelik, \emph{Global and exponential attractors for
  3-d wave equations with displacement dependent damping}, Math. Methods Appl.
  Sci. \textbf{29} (2006), no.~11, 1291--1306.

\bibitem{Pazy83}
Amnon Pazy, \emph{Semigroups of linear operators and applications to partial
  differential equations}, Applied Mathematical Sciences - Volume 44,
  Springer-Verlag, New York, 1983.

\bibitem{Segatti06}
Antonio Segatti, \emph{On the hyperbolic relaxation of the {C}ahn-{H}illiard
  equation in 3-d: Approximation and long time behaviour}, Math. Models Methods
  Appl. Sci. \textbf{to appear} (2006).

\bibitem{Tanabe79}
Hiroki Tanabe, \emph{Equations of evolution}, Pitman, London, 1979.

\bibitem{Temam88}
Roger Temam, \emph{Infinite-dimensional dynamical systems in mechanics and
  physics}, Applied Mathematical Sciences - Volume 68, Springer-Verlag, New
  York, 1988.

\bibitem{Vicente09}
Andr\'{e} Vicente, \emph{Wave equations with acoustic/memory boundary
  conditions}, Bol. Soc. Paran. Mat. \textbf{27} (2009), no.~1, 29--39.

\bibitem{Wu&Zheng06}
Hao Wu and Songmu Zheng, \emph{Convergence to equilibrium for the damped
  semilinear wave equation with critical exponent and dissipative boundary
  condition}, Quart. Appl. Math. \textbf{64} (2006), no.~1, 167--188.

\bibitem{Zelik04}
Sergey Zelik, \emph{Asymptotic regularity of solutions of singularly perturbed
  damped wave equations with supercritical nonlinearities}, Discrete Contin.
  Dyn. Syst. \textbf{11} (2004), no.~3, 351--392.

\bibitem{Zheng04}
Songmu Zheng, \emph{Nonlinear evolution equations}, Monographs and Surveys in
  Pure and Applied Mathematics - Volume 133, Chapman \& Hall/CRC, Boca Raton,
  2004.

\bibitem{Zheng&Milani05}
Songmu Zheng and Albert Milani, \emph{Global attractors for singular
  perturbations of the {C}ahn-{H}illiard equations}, J. Differential Equations
  \textbf{209} (2005), no.~1, 101--139.

\end{thebibliography}
\providecommand{\bysame}{\leavevmode\hbox to3em{\hrulefill}\thinspace}
\providecommand{\MR}{\relax\ifhmode\unskip\space\fi MR }
\providecommand{\MRhref}[2]{%
  \href{http://www.ams.org/mathscinet-getitem?mr=#1}{#2}
}
\providecommand{\href}[2]{#2}

\end{document}